\documentclass{amsart}[12pt]

\usepackage{amsfonts,amssymb,amsmath,amscd,amsthm, mathrsfs}
\usepackage{graphicx}
\usepackage{epstopdf}
\usepackage{hyperref}
\usepackage{upgreek}

\newtheorem{thm}{Theorem}[section]
\newtheorem{dfn}{Definition}[section]
\newtheorem{prop}{Proposition}[section]
\newtheorem{lm}{Lemma}[section]
\newtheorem{cor}{Corollary}[section]
\newtheorem{rem}{Remark}[section]

\newcommand{\mb}[1]{\mathbb{#1}}
\newcommand{\mf}[1]{\mathfrak{#1}}
\newcommand{\mc}[1]{\mathcal{#1}}
\newcommand{\rr}{\mathbb{R}}
\newcommand{\zz}{\mathbb{Z}}

\newcommand{\ey}{\frac{1}{2}}

\newcommand{\R}{\mf{R}}
\newcommand{\A}{\mc{A}}

\newcommand{\e}{\mathbf{e}}

\newcommand{\dl}{\delta}

\newcommand{\ep}{\varepsilon}

\newcommand{\Lmd}{\Lambda}
\newcommand{\lmd}{\lambda}
\newcommand{\om}{\omega}
\newcommand{\Om}{\Omega}
\newcommand{\sg}{\sigma}

\newcommand{\tht}{\theta}

\newcommand{\lmn}{\Lambda^n}
\newcommand{\hlm}{\hat{\Lambda}^n}

\newcommand{\qd}{\dot{q}}
\newcommand{\qey}{q^{\ep_1}}
\newcommand{\qee}{q^{\ep_2}}
\newcommand{\qt}{\tilde{q}}
\newcommand{\qo}{q^{\om}}
\newcommand{\qe}{q^{\ep}}
\newcommand{\yt}{\tilde{y}}
\newcommand{\xt}{\tilde{x}}
\newcommand{\qh}{\hat{q}}

\newcommand{\xf}{\mf{x}}
\newcommand{\yf}{\mf{y}}
\newcommand{\zf}{\mf{z}}
\newcommand{\xfd}{\dot{\mf{x}}}

\newcommand{\thd}{\dot{\tht}}

\newcommand{\xd}{\dot{x}}
\newcommand{\yd}{\dot{y}}
\newcommand{\omn}{\Om_{[n/2]}}

\newcommand{\I}{\mathbf{I}}
\newcommand{\J}{\mathbf{J}}
\newcommand{\N}{\mathbf{N}}

\newcommand{\en}{\frac{n}{2}}

\begin{document}

	\title[spatial double choreography]{Spatial double choreographies of the Newtonian $2n$-body problem}
	\author{Guowei Yu}
	\email{yu@ceremade.dauphine.fr}
	
	\address{University of Paris-Dauphine \& IMCCE, Paris Observatory\footnote{current address: Dipartimento di Matematica ``Giuseppe Peano'', Universit\`a degli Studi di Torino, Italy}}
	
	\begin{abstract} 
	    In this paper, for the spatial Newtonian $2n$-body problem with equal masses, by proving the minimizers of the action functional under certain symmetric, topological and monotone constraints are collision-free, we found a family of spatial double choreographies, which have the common feature that half of the masses are circling around the $z$-axis clockwise along a spatial loop, while the motions of the other half masses are given by a rotation of the first half around the $x$-axis by $\pi$.

	    Both loops are simple, without any self-intersection, and symmetric with respect to the $xz$-plane and $yz$-plane. The set of intersection points between the two loops is non-empty and contained in the $xy$-plane. The number of such double choreographies grows exponentially as $n$ goes to infinity.  
	\end{abstract}
	
	\maketitle
	
\section{Introduction} \label{sec:intro}

The Newtonian $N$-body problem describes the motion of $N$ point masses $m_i \in \rr^+$, $i \in \N =\{0, \dots, N-1\},$ according to Newton's law of universal gravity:
\begin{equation} \label{eqn:nbody}
m_i \ddot{q}_i = \frac{\partial}{\partial q_i} U(q), \quad \forall i \in \N,
\end{equation}
where $q=(q_i)_{i \in \N}$ ($q_i =(x_i, y_i, z_i) \in \rr^3$ represents the position of $m_i$) and $U(q)$ is the potential function (the negative potential energy)
$$ U(q) = \sum_{0 \le i < j \le N-1} \frac{m_i m_j}{|q_i -q_j|}. $$
Equation \eqref{eqn:nbody} is the Euler-Lagrange equation of the action functional 
\begin{equation}
\label{eqn:action} \A_{T_1, T_2}(q) = \int_{T_1}^{T_2} L(q(t), \qd(t)) \, dt, \;\; q \in H^1([T_1, T_2], \rr^{3N}),
\end{equation}
where $H^1([T_1, T_2], \rr^{3N})$ is the space of Sobolev paths and $L(q, \qd)$ is the associated Lagrangian
$$ L(q, \qd) = K(\qd)+U(q), \;\; K(\qd) =\ey \sum_{i=0}^{N-1} m_i|\qd_i|^2 . $$
For simplicity, we set $\A_T(q)= \A_{0, T}(q)$, for any $T>0$. 

If $q \in H^1([T_1, T_2], \rr^{3N})$ is a collision-free critical point of the action functional $\A_{T_1, T_2}$, then it is a smooth solution of equation \eqref{eqn:nbody}. By \emph{collision-free}, we mean $q(t) \in \rr^{3N} \setminus \Delta$, for any $t \in [T_1, T_2]$, where $\Delta$ is the set of collision configurations
$$ \Delta =\{q =(q_i)_{i \in \N} \in \rr^{3N}| \; q_{i_1} =q_{i_2}, \text{ for some } i_1 \ne i_2 \in \N \}. $$

However due to the weakness of Newtonian potential, the existence of collision along a path does not mean its action value must be infinity. This was already noticed by Poincar\'e. As a result, an action minimizer may contain collision. Because of this, for a long time variational methods were not very useful in this classic problem. The breakthroughs were the proofs of the Hip-Hope solution \cite{CV00} of the four body problem and the Figure-Eight solution \cite{CM00} of the three body problem by action minimization method. The key idea behind both proofs is the invariance of the Lagrangian under the permutation of equal masses, which allows the authors to impose symmetric constraints on the paths. Since then, many new periodic and quasi-periodic solutions have been found using this method. It is impossible to give a complete list, to name a few of them, see \cite{BFT08}, \cite{BT04}, \cite{FT04}, \cite{Fe06}, \cite{Sh06} and the references with in . 

We briefly recall the idea of symmetric constraints in the following, for details see \cite{FT04}. Let $\Lmd= H^1(\rr / T\zz, \rr^{3N})$ be the space of $T$-periodic Sobolev loops and $G$ a finite group with its action on the loop space $\Lmd$ defined as follows
$$ g\big(q(t)\big)= \big(\rho(g)q_{\sg (g^{-1})(0)}, \dots, \rho(g) q_{\sg (g^{-1})(N-1)}\big) \big(\uptau(g^{-1})t \big), \; \forall g \in G, $$
where
\begin{enumerate}
 \item[(a).] $\uptau: G \to O(2)$ representing the action of $G$ on the time circle $\rr/ T\zz$;
 \item[(b).] $\rho: G \to O(3)$ representing the action of $G$ on $\rr^3$;
 \item[(c).] $\sg: G \to \mc{S}_{\N}$ representing the action of $G$ on the index set $\N$, where $\mc{S}_{\N}$ is the permutation group of $\N$. 
\end{enumerate}
$\Lmd^G = \{ q \in \Lmd|\; g(q(t)) = q(t), \; \forall g \in G \}$ is the space of \emph{$G$-equivariant loops}. If the masses satisfy the following condition:
\begin{equation}
\sg(g)i = j, \; \text{ for some } g \in G \text{ and } i, j \in \N  \Rightarrow m_i = m_j,
\end{equation}
then the action functional $\A$ is invariant under the group action. By Palais' symmetric principle \cite{Pa79}, a critical point of $\A$ in $\Lmd^{G}$ is a critical point of $\A$ in $\Lmd$ as well. 

Based on Marchal's average method (\cite{Mc02}, \cite{C02}), Ferrario and Terracini \cite{FT04} proved any local minimizer of the action functional in $\Lmd^{G}$ must be collision-free, if the group action of $G$ satisfies the \emph{rotating circle property}. 

Compare to the idea of imposing symmetric constraints described as above, an older idea, at least goes back to Poincar\'e, is to impose topological constraints on the path or loop space, see \cite{Mo98} and \cite{C02}. However when topological constraints are involved (with or without additional symmetric constraints), it is much harder to show a corresponding action minimizer is collision-free. For one reason, Marchal's average method or the \emph{rotating circle property} does not work, and in some cases it has been proven that an action minimizer does contain collision, see \cite{Go77}, \cite{Ve01} and \cite{Mo02}. Because of this, fewer results are available along this line, see \cite{Ch08}, \cite{Ch13}, \cite{CL09}, \cite{FGN11}, \cite{Sh14}, \cite{WZ16} and \cite{Y15c}.

Among all the solutions found by minimization methods in the last fifteen years, a particularly interesting family is the \emph{simple choreographies}: a simple choreography is a periodic solution of the $N$-body problem with all the masses travel on a single loop in $\rr^2$ or $\rr^3$. Here we do not consider those solutions that are simple choreographies in some rotating coordinates.  

Although a lot of simple choreographies have been found numerically by Sim\`o \cite{Si00} and others. Rigorous proofs are only available for some of them: the rotating $n$-gon \cite{BT04}, the Figure-Eight of three body \cite{CM00}, the Figure-Eight type of odd bodies \cite{FT04} and the Super-Eight of four body \cite{Sh14}. All these examples belong to the sub-family called \emph{linear chains} (\cite{CGMS02}), which looks like a sequence of consecutive bubbles. A variational proof of the linear chains has been established by the author recently in \cite{Y15c}. There are also simple choreographies not from the family of linear chains, for example see \cite{OX15}.

All these solutions are proven by action minimization method under the assumption that all the masses are equal. It is still an open problem whether there are simple choreographies with unequal masses \cite{C04}. In the rest of the section, we assume all the masses are equal and for simplicity set $m_i=1$, $\forall i \in \N$. 

We remark that all the simple choreographies mentioned above that have been proven, are \emph{planar solutions}, i.e., the masses always stay on a fixed plane in $\rr^3$. As a contrary, there are also \emph{spatial solutions}, where the masses do not always contained in any fixed plane of $\rr^3$. 

In general, when we are considering a minimization problem in $\rr^3$, it is not an easy task to determine whether an action minimizer is planar or spatial, see \cite{Ch07}. For example, the action of the dihedral group $D_6$ of the Figure-Eight of the planar three-body problem can be extended to $\rr^3$, however so far no proof is available regarding whether the corresponding action minimizer is planar or spatial, see \cite{CFM05} or \cite{Fe06}. The same problem also occurs if we try to extend the minimization problem considered in \cite{Y15c} to $\rr^3$. Up to our knowledge, this problem can be solved in two special cases: (i) the minimizer is a relative equilibrium; (ii) the minimizer is found in certain uniform rotating coordinates, see \cite{BT04}, \cite{C05}, \cite{CF09}, \cite{CFM05}, \cite{Fe06} and \cite{TV07}.

In this paper we try to demonstrate that from a given planar simple choreographies, there may be a systematic way of constructing an entire family of spatial \emph{double choreographies} (a double choreography is a periodic solution with the masses traveling on two geometrically different loops.). First we give an intuitive idea as follows:


{\em Let's assume there are $n$ masses moving along a given simple choreography lying in a plane parallel to and above the $xy$-plane (we may also fix the center of mass on the $z$-axis). Now assume there are another $n$ masses whose motion is given by a rotation of $\pi$ around the $x$-axis of the first $n$ masses. This gives a \textbf{fake} spatial double choreography. By \textbf{fake}, we mean one has to ignore the gravitational force between the two $n$-body systems, otherwise such a double choreography can't exist. 

From a variational point of view the above solution can not exist either, because the action functional is not coercive in the above setting. One just needs to separate the two parallel planes where the $n$-body sub-systems lying further and further away from each other along the $z$-axis, until their distance becomes infinity. 

One way to get coercivity is by assuming the masses go above and below the $xy$-plane from time to time, then after a rotation of $\pi$ around the $x$-axis, the two $n$-body sub-systems will tangle with each other along the $z$-axis and force coercivity. From a dynamical point of view, this means there is a chance that the positions and the gravitational force between the sub-systems could be carefully balanced, so the original simple choreographies of each $n$-body system may be deformed, but still kept, and gives us a real spatial double choreography of the $2n$-body problem.}

At a first glance, such a delicate balance seems hard to achieve. However by considering a minimization problem with proper symmetric and topological constraints, we establish the existence of such spatial double choreographies when the corresponding simple choreography is the simplest one: the rotating $n$-gon. Furthermore we can more or less control how the two loops tangle each other by imposing different topological constraints, and as a result the number of such spatial double choreographies grows exponentially as $n$ goes to infinity. We believe the idea could work for more general simply choreographies, at least those from the family of linear chains. 

Following the above approach, in this paper we show the existence of a family of spatial double choreographies having the common feature that each double choreography consists of two spatial loops (each without any self-intersection) identical to each other after a rotation of $\pi$ around the $x$-axis, and if we project it to the $xy$-plane, half of the masses circles around the origin clockwise along a loop, which is symmetric with respect to the $x$-axis and $y$-axis and does not have any self-intersection (like a slightly deformed circle), while the other half masses circle around counter-clockwise on the same loop, see Figure \ref{fig:even} and \ref{fig:odd} for illuminating pictures. Up to our knowledge such spatial double choreographies have not been found numerically or proven analytically before. 

\begin{figure}
  \centering
  \includegraphics[scale=0.75]{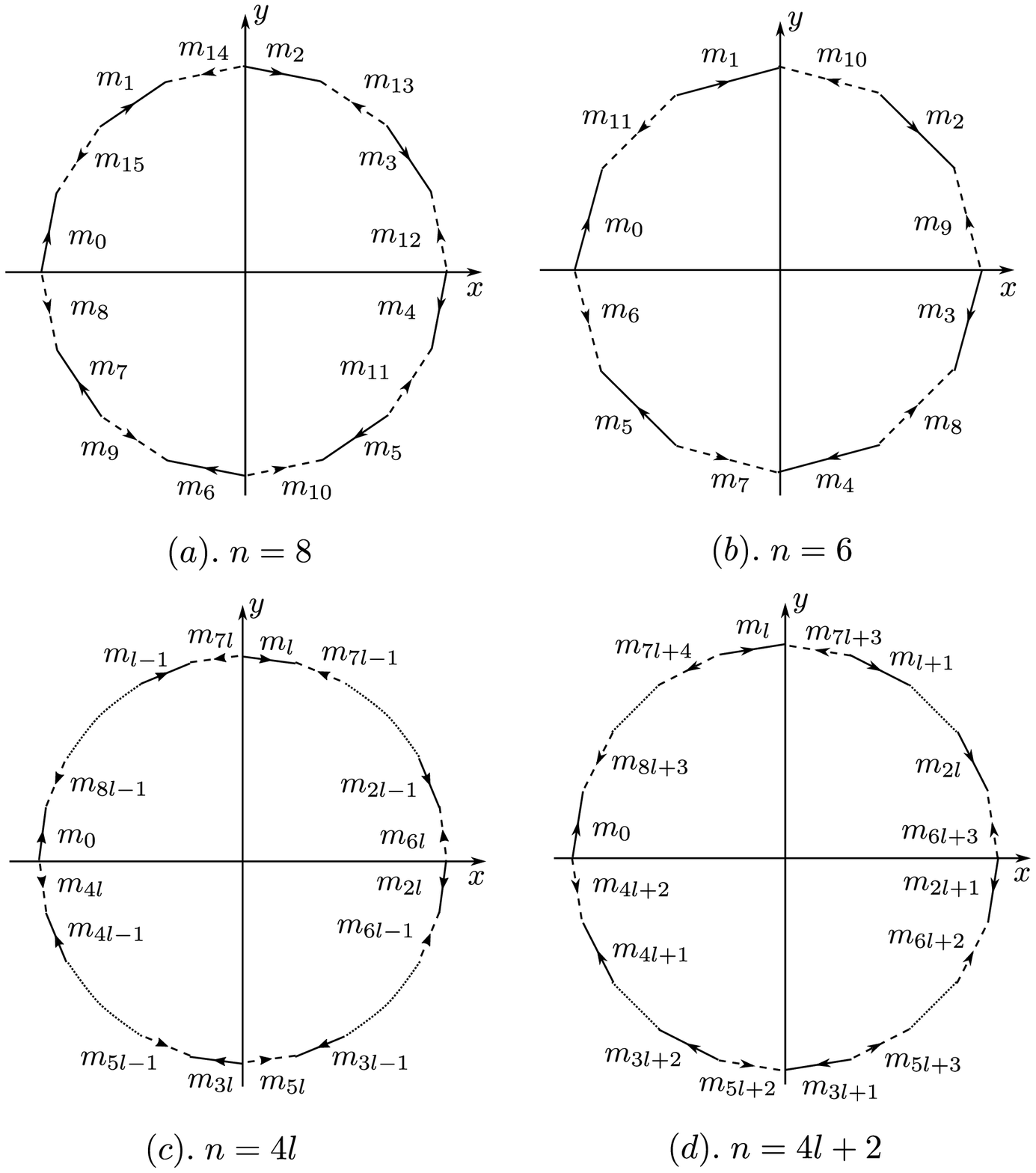}
  \caption{}
  \label{fig:even}
\end{figure}

\begin{figure}
  \centering
  \includegraphics[scale=0.75]{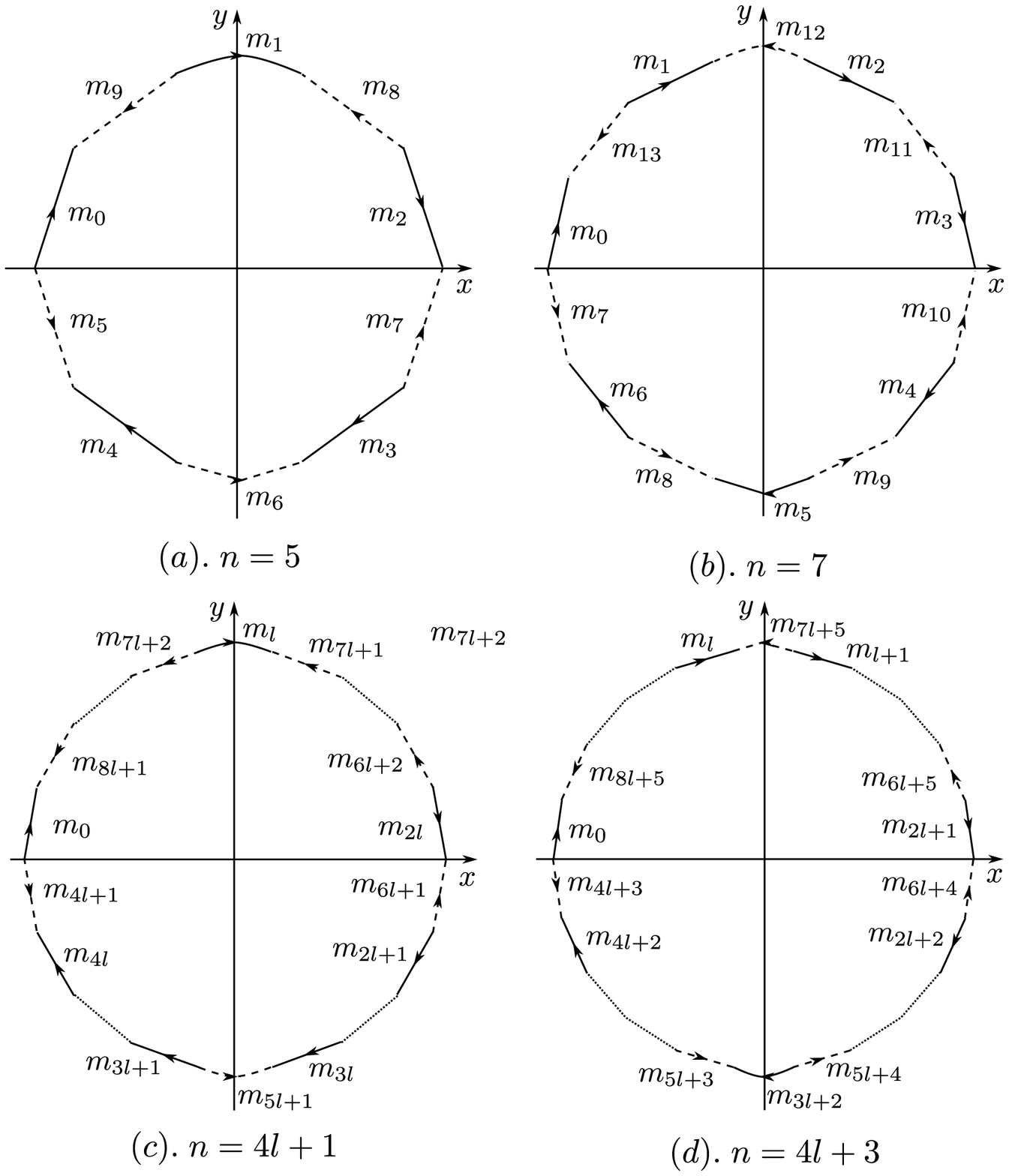}
  \caption{}
  \label{fig:odd}
\end{figure}

In the following, we assume $N=2n$ and $n \ge2$. Because of the homogeneity of the potential, if $q(t)$ is a solution of \eqref{eqn:nbody}, so is $\lmd^{-2/3}q(\lmd t)$ for any $\lmd >0$. To simplify notation, we set the time-period $T=n$ and correspondingly $\Lmd = H^1(\rr/n \zz, \rr^{6n})$. First we will introduce the proper symmetric constraints.   

Consider the finite group $G_n=D_n \times \zz_2 \times \zz_2$, where
$$ D_n = \langle g_1, g_2 |\; g_1^n= g_2^2 =1, (g_1 g_2)^2 =1 \rangle; $$ 
$$ \zz_2 \times \zz_2 = \langle h_1 | \; h_1^2 =1 \rangle \times \langle h_2 | \; h_2^2 =1 \rangle.$$
Let `Id' be the identity,  $\mf{R}_{xz}$ the reflection with respect to the $xz$-plane and $\mf{R}_x$ a rotation of $\pi$ around the $x$-axis ($\mf{R}$ with other sub-indices will be defined similarly). Furthermore $[\cdot]$ will denote the integer part of a number. With these notations, we define the action of $G_n$ as follows: 
$$ \uptau(g_1)t=t, \;\; \rho(g_1)= \text{Id}, \;\; \sg(g_1) =(0,1, \dots, n-1)(n, n+1, \dots, 2n-1);$$
$$ \uptau(g_2)t= 1-t, \;\; \rho(g_2)=\R_{xz}, \;\; \sg(g_2)= \prod_{i=0}^{[\frac{n-1}{2}]} \big((i, n-1-i)(n+i, 2n-1-i)\big);$$
$$ \uptau(h_1)t=t, \;\; \rho(h_1) = \R_x, \;\; \sg(h_1)= \prod_{i=0}^{n-1}(i, n+i);$$
the action of $h_2$ will be defined differently for even and odd $n$,\\
if $n$ is even,
$$ \uptau(h_2)t= t, \;\; \rho(h_2) = \R_z, \;\; \sg(h_2)= \prod_{i=0}^{\frac{n}{2}-1} \big((i, \frac{n}{2}+i)(n+i, \frac{3}{2}n+i)\big);$$
if $n$ is odd, 
$$ \uptau(h_2)t= \ey-t, \;\; \rho(h_2)= \R_{yz}, $$
$$ \sg(h_2) = \prod_{i=0}^{[\frac{n}{4}]}\big((i, [\frac{n}{2}]-i)(n+i, n+[\frac{n}{2}]-i)\big) \prod_{i=1}^{[\frac{n+1}{4}]}\big(([\frac{n}{2}]+i, n-i)(n+[\frac{n}{2}]+i, 2n-i)\big).$$



For simplicity, we set $\Lmd^n = \Lmd^{G_n}$. Due to the symmetric constraints, the following equations hold for any $q \in \lmn$ and $t \in \rr$, 
\begin{align}
\label{eqn:g1} g_1 \; & \Rightarrow \; \begin{cases}
					q_i(t) &= q_0(t+i), \\
					q_{i+n}(t) & = q_n(t+i), \\
\end{cases}  \quad \forall i \in \{0, \dots, n-1\},\\
\label{eqn:h1} h_1 \; & \Rightarrow \; q_{i+n}(t) = \R_x q_i(t),  \quad \quad \quad \; \forall i \in \{0, \dots, n-1\}, \\
\label{eqn:g1g2} g_1, g_2 \; & \Rightarrow \; q_0(t) = \R_{xz} q_0(n-t), \\
\label{eqn:g1h2} g_1, h_2 \; & \Rightarrow \;  q_0(t) = \begin{cases}
                              \R_z q_0(n/2 +t), \; & \text{ if } n \text{ is even, } \\
                              \R_{yz}q_0(n/2 -t), \; & \text{ if } n \text{ is odd. }
                             \end{cases}
\end{align}

\begin{rem}  We briefly explain the implication of the above equations:
\begin{enumerate}
\item[(a).] by \eqref{eqn:g1}, $m_i$ moves along the loop $q_0(\rr/n\zz),$ if $i \in \{0, \dots, n-1\}$, and along the loop $q_n(\rr/n\zz),$ if $i \in \{n, \dots, 2n-1\}$;
\item[(c).] by \eqref{eqn:h1}, $q_0(\rr/n\zz)$ and $q_n(\rr/n\zz)$ are identical to each other after a rotation of $\pi$ around the $x$-axis;
\item[(b).] by \eqref{eqn:g1g2} and \eqref{eqn:g1h2}, $q_0(\rr/n\zz)$ is symmetric with respect to $xz$-plane, $yz$-plane and the origin, in particular
\begin{equation}
 \label{eqn:q0symmetry} y_0(0) = x_0(n/4) = y_0(n/2) = x_0(3n/4)=0.
\end{equation}
\end{enumerate}
\end{rem}

When $n$ is even, $[0,1/2]$ is a \emph{fundamental domain} of the $G_n$-equivariant loops (see \cite{FT04}), i.e., once $q(t)$, $t \in [0,1/2]$ is defined, the entire $G_n$-equivariant loop $q(t), t \in \rr/n \zz$ is uniquely determined through symmetries. When $n$ is odd, $[0, 1/4]$ is a fundamental domain. Because of this, there is a one-to-one map between $\lmn$ and the following two sets
\begin{align*}
\{q(t), t \in [0,1/2]| \; q \in \lmn\}, & \text{ if } n \text{ is even};\\
\{q(t), t \in [0,1/4]| \; q \in \lmn \}, & \text{ if } n \text{ is odd}. 
\end{align*}
By abuse of notation, these sets of paths will also be denoted by $\lmn$. We point out that $q \in \Lmd^n$ is also uniquely determined through symmetries, once $q_0(t), t \in [0, n/4]$ is given. This point of view will also be useful later. 

As we mentioned before, the action functional $\A$ is not coercive in $\lmn$. To force coercivity, we need the two $n$-body sub-systems to tangle with each other along the $z$-axis and this can be done by adding topological constraints into the problem. However this makes it more difficult to rule out collision in the action minimizers. To overcome this difficulty, an additional type of constraints called \emph{monotone constraints} were introduced by the author in \cite{Y15c}. In the following, we will first introduce the monotone constraints, as it will be easier to impose the topological constraints afterwards.

\begin{dfn}
\label{dfn:smc} We set $ \Lmd^n_< = \{ q \in \Lmd^n | \; q \text{ is \textbf{strictly $x$ and $y$-monotone}} \}.$\\
$q \in \Lmd^n$ is \textbf{strictly $x$-monotone}, if $ x_0(t_1) < x_0(t_2), \;\; \forall 0 \le t_1 < t_2 \le n/4;$\\
$q \in \Lmd^n$ is \textbf{strictly $y$-monotone}, if $ y_0(t_1) < y_0(t_2), \;\; \forall 0 \le t_1 < t_2 \le n/4.$
\end{dfn}
Let $\mf{Q}_2$ denote the second quadrant of $\rr^3$, i.e.,
$$ \mf{Q}_2 = \{ (x,y,z) \in \rr^3| \; x \le 0, \; y \ge 0 \},$$
and the other three quadrants are defined as follows
$$ \mf{Q}_1 = \mf{R}_{yz} \mf{Q}_2, \;\; \mf{Q}_3 = \mf{R}_{xz} \mf{Q}_2, \;\; \mf{Q}_4 = \mf{R}_{xz} \mf{Q}_1.$$
If $q \in \lmn_<$, being strictly $x$ and $y$-monotone implies $q_i(t) \in \mf{W}_i(q)$, for any $i \in \N$ and $t \in [0, 1/2]$, where 
\begin{equation}
\label{eqn:well1} \mf{W}_i(q) = \{ (x,y,z) \in \rr^3| \big(x-x_i(0)\big)\big(x-x_i(\ey)\big) \le 0, \; \big(y-y_i(0)\big)\big(y-y_i(\ey)\big) \le 0 \},
\end{equation}
except for $i = [n/4]$ or $n+[n/4]$, when $n$ is odd, in which case
\begin{equation}
\label{eqn:well2} \mf{W}_i(q) = \{ (x,y,z)\in \rr^3| \big(x-x_i(0)\big)\big(x-x_i(\ey)\big) \le 0, \;\big(y-y_i(0)\big)\big(y-y_i(\frac{1}{4})\big) \le 0 \},
\end{equation}
because in the later cases, the mass $m_i$ moves across the $yz$-plane at $t =1/4$ and changes direction along the $y$-axis. 

Geometrically each $\mf{W}_i(q)$ looks like an infinite well (in the shape of square) with a non-empty interior $\mathring{\mf{W}}_i(q) \ne \emptyset$. Furthermore $\mathring{\mf{W}}_{i_1}(q) \cap \mathring{\mf{W}}_{i_2}(q) = \emptyset$, for any $i_1 \ne i_2.$ As a result, $q(t)$ is collision-free for any $t \in (0, 1/2)$ and the only possible collisions are binary ones at $t=0$ or $1/2$. Moreover due to the symmetric constraints, a binary collision can happen if and only if
$$ z_0(i/2) = 0, \text{ for some } i \in \N \text{ or equivalently } i \in \{0, \dots, [n/2] \}.$$

The symmetric and monotone constraints introduced above tell us quite a lot of the geometrical information of the loop $q_0(t), t \in \rr/n \zz$:
\begin{enumerate}
 \item[(i).] $q_0(t) \in \mathring{\mf{Q}}_2$, $\forall t \in (0, n/4),$ where $\mathring{\mf{Q}}_2$ is the interior of $\mf{Q}_2$; 
 \item[(ii).] As a loop $q_0$ is symmetric with respect to the $xz$ and $yz$-plane;
 \item[(iii).] As a loop $q_0$ does not have any self-intersection. 
  \end{enumerate} 

Pictures in Figure \ref{fig:even} and \ref{fig:odd} indicate the motion of the masses from $t=0$ to $t=1/2$ after projecting to the $xy$-plane. As we can see from the pictures, at $t=0$ or $1/2$, to avoid a binary collisions between a corresponding pair of masses, we need one of them to go above the $xy$-plane and the other to go below. This gives us an idea of how to impose the proper topological constraints.

Let $\hlm=\{ q \in \lmn|\; q(t) \notin \Delta, \; \forall t \in \rr/n \zz \}$ be the set of all collision-free $G_n$-equivariant loops, and $\hlm_{<} = \lmn_{<} \cap \hlm$. By the above explanation
\begin{equation}
\label{eqn:coll-free}
\hlm_< = \{ q \in \Lmd^n_<: z_0(i/2) \ne 0, \; \forall i =0, \dots, [n/2] \}.
\end{equation}

\begin{dfn}
\label{dfn:tc} 
Let $\omn$ be a subset of $\{1,-1\}^{[\frac{n}{2}]+1}$ defined as 
$$ \Om_{[n/2]} = \big\{ \om=(\om_i)_{i=0}^{[\frac{n}{2}]}| \; \om_{i_1} \ne \om_{i_2}, \text{ for some } 0 \le i_1 < i_2 \le [n/2] \big\}.$$
For any $\om \in \omn$, we say $ q \in \hlm_{<}$ satisfies the \textbf{$\om$-topological constraints}, if the following hold
\begin{equation}
\label{eqn:tc} z_0(i/2) = \om_i |z_0(i/2)|, \;\; \forall i =0, \dots, [n/2],
\end{equation}
and set 
$$  \hlm_{<, \om}= \{ q \in \hlm_< | \; q \text{ satisfies the } \om\text{-topological constraints} \}. $$
\end{dfn}

\begin{rem}
The values of $\om_i$ in $\om$ determines whether $m_0$ is above or below the $xy$-plane at the moments $t \in \{i/2| \; 0 \le i \le [n/2] \}$. By choosing different $\om$, we determine how the two loops in the double choreography tangle each other along the vertical direction. 
\end{rem}

Clearly $\hlm_{<, \om}$ is not a closed set. To be able to use the standard argument from calculus of variation, we let $\lmn_{\le, \om}$ be the weak closure of $\hlm_{<, \om}$ in $\lmn$ (an alternative definition of $\lmn_{\le, \om}$ will be given in Definition \ref{dfn:lmn}). The following theorem is the main result of our paper.
\begin{thm}
\label{thm:main} For any $\om \in \Om_{[n/2]}$ satisfying the following condition 
\begin{equation}
\label{eqn:odd om} \text{ when } n \text{ is odd, there exist } \{i_1 \ne i_2\} \subset [1, [n/2]] \cap \zz, \text{ such that } \om_{i_1} \ne \om_{i_2}.
\end{equation}
There exits at least one collision-free $\qo \in \hat{\Lmd}^n_{<, \om}$, which is a minimizer of the action functional $\A_n$ in $\lmn_{\le, \om}$. Furthermore $\qo$ is spatial double choreography of the $2n$-body problem satisfying 
\begin{align}
\label{eq: x dot >0} \xd_0^{\om}(0)=0, \; \text{ and } \; \xd_0^{\om}(t) >0, \; & \forall t \in (0, n/4], \\
\label{eq: y dot >0} \yd_0^{\om}(n/4)=0, \; \text{ and } \; \yd_0^{\om}(t) >0, \; & \forall t \in [0, n/4). 
\end{align}
\end{thm}

\begin{rem}
Since $\qo \in \hat{\Lmd}^n_{<, \om}$, it satisfies \eqref{eqn:g1}, \eqref{eqn:h1}, \eqref{eqn:g1g2} and \eqref{eqn:g1h2}. As a result, it is a double choreography with the two loops symmetric with respect to the $xz$-plane and $yz$-plane, and identical to each other after a rotation of $\pi$ around the $x$-axis; being strictly $x$ and $y$-monotone means it must be a spatial solution and each of the two loops can not have any self-intersection; the $\om$-topological constraints guarantee the set of intersection points between the two loops is non-empty and the symmetric constraints imply they must be contained in the $xy$-plane.
\end{rem}

\begin{rem}
 As we can see, the number of $\om \in \Om_{[n/2]}$ satisfying condition \eqref{eqn:odd om} grows exponentially as $n$ goes to infinity. Therefore so is the number of different double choreographies obtained by Theorem \ref{thm:main} (two double choreographies are different, if we cannot identify one to another by a time re-parameterization or space orthogonal transformation).  
\end{rem}

\begin{rem}
We point out that when $n=3$, there is no $\om \in \Om_{[3/2]}$ satisfying condition \eqref{eqn:odd om}, as in this case $[1, [3/2]] \cap \zz = \{1\}$. Condition \eqref{eqn:odd om} is needed, because for our approach to work the degenerate cases that the minimizer is entirely contained in the $xz$-plane or the $yz$-plane need to be excluded, see Lemma \ref{lm:dfm2}. When $n$ is odd, we can not rule out the possibility that the minimizer is entirely contained in the $xz$-plane without the condition given in \eqref{eqn:odd om}, see Lemma \ref{lm:nondeg} and \ref{lm:B3}.  
\end{rem}

It will be interesting to compare our result with those obtained in \cite{FGN11} and \cite{WZ16}, where the authors also considered minimization problems with mixed symmetric and topological constraints, and proved the existence of certain spatial double choreographies as collision-free minimizers similar to ours. When $n=2$, the solution obtained in Theorem \ref{thm:main} are likely to be the same as those obtained in \cite[Theorem 2.1]{FGN11} and \cite[Theorem 3]{WZ16} with the proper topological constraints.

One of the key ideas in the above two papers that allows them to handle the topological constraints is by imposing \emph{strong symmetric constraints}. By \emph{strong symmetric constraints}, we mean at any moment once the position of one mass is known, then the positions of all the other masses are uniquely determined by symmetries. In some sense, this makes the problem equivalent to the motion of one point mass in some singular potential field. 

While the \emph{strong symmetric constraints} helped simplify the problem, it also restricted the possible choices of topological constraints. For example in our terminology, roughly speaking the solutions found in \cite{WZ16} all correspond to $\om$-topological constraints with $\om_i \ne \om_{i+1},$ for each $i$. 

From this point of view our symmetric constraints is much weaker, as at any moment (except the boundaries of the fundamental domains), we need to know the positions of least $n/2$ masses, when $n$ is even, or $n$ masses, when $n$ is odd, to be able to determine the positions of all the other masses. As a result the idea used in \cite{FGN11} and \cite{WZ16} to rule out collisions will not work in our setting. 

The technical lemmas we use to rule out collisions are generalizations of ideas introduced in \cite{Y15c}. However instead of the monotone constraints that required in \cite{Y15c}, here we prove them under much weaker condition (see Definition \ref{dfn:sepa}) and more importantly these lemmas will be proven for arbitrary choice of masses, which gives them the potential to be applied in more general problems, particularly problems not all masses are equal. 

\textbf{Notations:} The following notations will be applied through the paper:
\begin{itemize}
\item $l$ and $\ell$ are used as two different letter;
\item $\e_1, \e_2$ and $\e_3$ represent the three unit vectors $(1,0,0), (0,1,0)$ and $(0,0,1)$ in $\rr^3$;
\item if $q \in H^1([T_1, T_2], \rr^{3n})$, then $q_i \in H^1([T_1, T_2], \rr^3)$ is the corresponding path of $m_i$ and $x_i, y_i, z_i \in H^1([T_1, T_2], \rr)$ are the projection of $q_i$ to the $x$, $y$ and $z$-axis; 
\item if instead of $q$, a path is denoted by $q^{\om}, q^{\ep}, q^{\ep_1}, q^{\ep_2}, \qh, \qt$ or $\qt^{\ep_1}$, then the corresponding changes will be made on the notations $q_i, x_i, y_i, z_i$. 
\item Given any two non-negative integers $i_0, i_1$: 
$$ \{i_0, \dots, i_1 \} := \begin{cases}
\{i \in \zz: \; i_0 \le i \le i_1\}, \; & \text{ if } i_0 \le i_1, \\
\emptyset, \; & \text{ if } i_0 > i_1;
\end{cases} $$
\item $C$ and $C_i$, $ i \in \zz^+ $, represent different positive constants.
\end{itemize}

\section{Deformation Lemmas Involving Binary Collision} \label{sec:lemma}

The purpose of this section is to prove several deformation lemmas that can be used to rule out binary collisions in an action minimizer under certain symmetric and topological constraints. In this section instead of $N =2n$, we only assume $N \ge 3$. Furthermore the condition that all masses must be equal will also be dropped (one can have an arbitrary choice of positive masses). We believe results from this section may be used to rule out binary collisions in problems with topological constraints besides the one we are considering here. 

For the rest of this section, let $q:[0,T] \to \rr^{3N}$ be a collision solution, satisfying the following conditions:
\begin{enumerate}
 \item [(i).] for any $t \in (0, T)$, $q(t)$ is collision-free and satisfies equation \eqref{eqn:nbody};
 \item [(ii).] there is a isolated binary collision between $m_j$ and $m_k$ at $t=0$, i.e., 
 $$ q_{j}(0)= q_k(0) \ne q_i(0), \; \forall i \in \N \setminus \{j,k\}. $$ 
\end{enumerate}
Without loss of generality, let's assume $q_j(0)=q_k(0)=0$. Notice that there may be other isolated collisions at $t=0$, besides the one between $m_j$ and $m_k$. 

Let $q_c(t)= \frac{m_jq_j(t) + m_kq_k(t)}{m_j+m_k}$ be the center of mass of $m_j$ and $m_k$, and
\begin{equation}
 \label{eqn:rel} \mf{q}_i(t) = (\xf_i(t), \yf_i(t), \zf_i(t)) = q_i(t) - q_c(t), \; \forall i \in \{j, k\},
\end{equation}
the relative position of $m_i$ with respect to $q_c(t)$. Put $\mf{q}_i$ in spherical coordinates $(r, \phi, \tht)$ with $r \in [0, +\infty)$, $\phi \in [0, \pi]$ and $\tht \in [0, 2\pi)$, we get
\begin{equation} \label{eqn:sph co}
\xf_i  = r_i \sin \phi_i \cos \tht_i, \;\; \yf_i = r_i \sin \phi_i \sin \tht_i, \;\; \zf_i = r_i \cos \phi_i.
\end{equation}
Notice that $m_j\mf{q}_j(t) +m_k\mf{q}_k(t) =0$, and this implies
$$ m_k r_k(t) = m_j r_j(t), \;\; \phi_k(t) = \pi - \phi_j(t), \;\; \tht_k(t) = \pi + \tht_j(t). $$

The following asymptotic estimates are the key to our study of isolated binary collisions.    
\begin{prop}
\label{prop:sundman} For $i \in \{j, k\}$ and $t>0$ small enough, 
$$ r_i(t) = C_1t^{\frac{2}{3}}+o(t^{\frac{2}{3}}), \;\; \dot{r}_i(t) = C_2 t^{-\frac{1}{3}} +o(t^{-\frac{1}{3}}).$$
\end{prop}
This is the well-known Sundman's estimate, for a proof see \cite{FT04}.
\begin{prop} \label{prop:angle}
For $i \in \{j, k\}$, there exist $\phi^+_i \in [0, \pi]$ and $\tht^+_i \in \rr$ satisfying
\begin{enumerate}
\item[(a).] $\lim_{t \to 0^+} \phi_i(t) = \phi_i^+, \; \lim_{t \to 0^+} \tht_i(t) = \tht_i^+$,
\item[(b).] $\lim_{t \to 0^+} \dot{\phi}_i(t) = \lim_{t \to 0^+} \dot{\tht}_i(t) = 0$,
\item[(c).] $\phi^+_k = \pi - \phi^+_j, \; \tht_k^+ = \pi + \tht^+_j.$
\end{enumerate}
\end{prop}

Proposition \ref{prop:angle} shows $m_j$ and $m_k$ approach to the binary collision from some definite directions. This is a well-known result, a proof of it for the planar case can be found in \cite[Proposition 4.2]{Y15b}. The spatial case can be proven similarly. 

Fix an arbitrary subset of indices $\I \subset \N$ with $\{j, k \} \subset \I$ for the rest of the section. In the deformation lemmas to be given below, we will focus on the $\I$-body problem, i.e., we only show how the paths of $m_i$, $i \in \I,$ should be deformed explicitly and leave the rest undefined. We are doing this because when applying our results to problems with symmetric constraints, the paths of $m_i, i \in \N \setminus \I,$ will be uniquely determined by the particular symmetric constraints associated to the problem. 

Let the corresponding kinetic energy, potential energies, Lagrangian and action functional of the $\I$-body problem be defined as follows:
$$ K_{\I}(\qd)= \ey \sum_{i \in \I} m_i|\qd_i|^2, \;\; U_{\I}(q) = \sum_{\{i_1 < i_2\} \subset \I} \frac{m_{i_1} m_{i_2}}{|q_{i_1} -q_{i_2}|}, $$
$$ L_{\I}(q, \qd) = K_{\I}(\qd) + U_{\I}(q), \;\; \A_{\I, T}(q) = \int_0^T L_{\I}(q, \qd) \, dt. $$

Our first deformation lemma is a local property in nature. It shows that in most cases after a local deformation near an isolated binary collision, we get a new path with lower action value and without the binary collision. 
\begin{lm}
\label{lm:dfm1} If $\phi^+_j \in [0, \pi)$ $($resp. $\phi^+_j \in (0, \pi]),$ then for $\ep_1 > 0$ and $t_0 = t_0(\ep_1)>0$ small enough, there is a $q^{\ep_1} \in H^1([0, T], \rr^{3N})$ $($a local deformation of $q$ near $t=0$ $)$, which satisfies $\A_{\I,T}(\qey) < \A_{\I,T}(q)$ and the following properties:
\begin{enumerate}
\item[(a).] if $i \in \I \setminus \{j, k\}$, $\qey_i(t) = q_i(t)$, $\forall t \in [0, T],$
\item[(b).] if $i \in \{j, k\},$
$$ \begin{cases}
\qey_i(t)= q_i(t), &\text{ when } t \in [t_0, T]; \\
|\qey_i(t) - q_i(t)| \le \ep_1, & \text{ when } t \in [0, t_0],
\end{cases}$$
\item[(c).] $\qey(t) \notin \Delta$,  $\forall t \in (0, t_0],$
\item[(d).] $\qey_{i_1}(0) \ne \qey_{i_2}(0)$, for any $i_1 \in \{j, k\}$ and $i_2 \in \I \setminus \{j, k\},$
\item[(e).] both $\qey_j(0)$ and $\qey_k(0)$ belong to the $z$-axis with $m_jz^{\ep_1}_j(0) = -m_k z^{\ep_1}_k(0) >0$ $($resp. $m_j z^{\ep_1}_j(0) =- m_k z^{\ep_1}_k(t) < 0),$ in particular $\qey_j(0) \ne \qey_k(0).$
\end{enumerate}
\end{lm}
\begin{rem}
In application, the condition $m_j z_j^{\ep_1}(0) > 0$ or $< 0$ is usually related with the topological constraints. 
\end{rem}
A detailed proof of this lemma in the planar case can be found in \cite[Proposition 4.3]{Y15b}. The spatial case can be proven similarly. We will not repeat it here. Roughly speaking based on the \emph{blow-up} technique introduced by Terracini, the proof of this lemma essentially relies on the following fact of the Kepler problem: the parabolic collision-ejection solution connecting two different points (with the same distance to the origin) has action value strictly large than the direct and indirect Keplerian arcs joining them (with the same transfer time). This result was attributed to C. Marchal in \cite{C05}. A proof can be found in \cite{FGN11}, \cite{Y15a} or \cite{Ch17}. 

Notice that Lemma \ref{lm:dfm1} does not hold, when $\phi_j^+ = \pi$ (resp. $\phi_j^+ =0$). In fact, by the result of Gordon in Kepler problem \cite{Go77}, a local deformation result like above does not exist under these conditions. 

To rule out the binary collision in the remaining cases, some global information has to be considered. For this purpose, we introduced the monotone constraints in \cite{Y15c}. Here we show the same result can be obtained under weaker conditions.

Let $\I_0, \I_1$ be two subsets of $\I$, which are not empty sets at the same time (we may have one of them being an empty set), satisfying the following:
$$ \I \setminus \{j, k \} = \I_0 \cup \I_1 \ne \emptyset, \;\; \I_0 \cap \I_1 =\emptyset. $$ 

\begin{dfn} \label{dfn:sepa}

We say $(q_i(t))_{i \in \I}, t \in [0,T],$ is \textbf{$x$-separated} $($by $m_j$ and $m_k)$, if 
\begin{enumerate}
\item[(i).] $\forall t \in [0, T]$, $x_j(T) \le x_j(t) \le x_j(0)=x_k(0) \le x_k(t) \le x_k(T),$
\item[(ii).] 
$\forall t \in [0, T]$, $x_i(t) \le x_j(T)$, if $i \in \I_0$; $x_i(t) \ge x_k(T)$, if $i \in \I_1$,
\end{enumerate}
and \textbf{$y$-separated} $($by $m_j$ and $m_k )$, if
\begin{enumerate}
\item[(iii).] $\forall t \in [0, T]$, $y_j(T) \le y_j(t) \le y_j(0)=y_k(0) \le y_k(t) \le y_k(T),$
\item[(iv).] 
$\forall t \in [0, T]$, $y_i(t) \le y_j(T)$, if $i \in \I_0$; $y_i(t) \ge y_k(T)$, if $i \in \I_1$.
\end{enumerate}
\end{dfn}
In the rest of the section, we always assume the $x$-separated or $y$-separated is by $m_j$ and $m_k$. 

First we need to show Lemma \ref{lm:dfm1} can be extended to the set of paths that are $x$ or/and $y$-separated, and this is demonstrated by the following deformation lemma. 
\begin{lm}
\label{lm:mc} Under the same notations and conditions of Lemma \ref{lm:dfm1}, for $\ep_1>0$ small enough, let $\qey \in H^1([0,T], \rr^{3N})$ be a local deformation of $q$ obtained through Lemma \ref{lm:dfm1}, 
\begin{enumerate}
\item[(a).] if $(q_i(t))_{i \in \I}, t \in [0, T]$, is $x$-separated, then there is a new $x$-separated path $\qt^{\ep_1}(t)=(\qt^{\ep_1}_i(t))_{i \in \I}, t \in [0, T]$, satisfying
\begin{equation}
 \label{eqn: A qt} \A_{\I, T}(\qt^{\ep_1}) \le \A_{\I, T}(\qey) < \A_{\I, T}(q),
 \end{equation} 
and 
\begin{equation}
\label{eqn:xjk monotone} \xt^{\ep_1}_j(t_1) \ge \xt^{\ep_1}_j(t_2), \;\; \xt^{\ep_1}_k(t_1) \le \xt^{\ep_1}_k(t_2), \;\; \forall 0 \le t_1 < t_2 \le T; 
\end{equation}
\item[(b).] if $(q_i(t))_{i \in \I}, t \in [0, T]$, is $y$-separated, then there is a new $y$-separated path $\qt^{\ep_1}(t)=(\qt^{\ep_1}_i(t))_{i \in \I}, t \in [0, T]$, satisfying \eqref{eqn: A qt} and 
\begin{equation}
\label{eqn:yjk monotone} \yt^{\ep_1}_j(t_1) \ge \yt^{\ep_1}_j(t_2), \;\; \yt^{\ep_1}_k(t_1) \le \yt^{\ep_1}_k(t_2), \;\; \forall 0 \le t_1 < t_2 \le T; 
\end{equation}
\item[(c).] if $(q_i(t))_{i \in \I}, t \in [0, T]$, is $x$ and $y$-separated, then there is a new $x$ and $y$-separated path $(\qt^{\ep_1}_i(t))_{i \in \I}, t \in [0, T]$, satisfying \eqref{eqn: A qt}, \eqref{eqn:xjk monotone} and \eqref{eqn:yjk monotone}. 
\end{enumerate}
\end{lm}
\begin{rem}
 The additional properties \eqref{eqn:xjk monotone} and \eqref{eqn:yjk monotone} are proven here, because they will be needed when the paths are required to be $x$ or/and $y$-monotone, but not just $x$ or/and $y$-separated. 
\end{rem}

\begin{proof}
We only give the details for $(q_i(t))_{i \in \I}, t \in [0, T]$, being $x$-separated, while the others are similar. 

Recall that $q_j(0)=q_k(0)=0$, so
$$ x_j(t) \le x_j(0) =0 = x_k(0) \le x_k(t), \;\; \forall t \in [0, T].$$

By the properties listed in Lemma \ref{lm:dfm1}, condition (ii) in Definition \ref{dfn:sepa} still holds for $q^{\ep_1}$. However we don't know if condition (i) in the same definition will still hold for it. Let $t_0=t_0(\ep_1)$ be the small positive number given in Lemma \ref{lm:dfm1}, we set 
$$ \dl_j = \max \{ x_j^{\ep_1}(t) | \; t \in [0, t_0] \}, \;\; \dl_k = -\min \{ x_k^{\ep_1}(t) | \; t \in [0, t_0]\}. $$
By Lemma \ref{lm:dfm1}, $x^{\ep_1}_j(0)= x^{\ep_1}_k(0)=0$. Therefore $\dl_j \ge 0$ and so is $\dl_k$. We further set 
$$ t_j= \min\{ t \in [0, t_0] | \; x^{\ep_1}_j(t) =\dl_j \}, \;\; t_k = \min \{t \in [0, t_0] | \; x^{\ep_1}_k(t) =-\dl_k \}, $$
and 
$$ \mb{T}_j = \{ t \in [0, t_j] | \; x_j^{\ep_1}(t) >0 \}, \;\; \mb{T}_k = \{ t \in [0, t_k] | \; x_k^{\ep_1}(t) <0 \}. $$

Notice that if $\dl_i=0$, $i \in \{j, k\}$, then $t_i=0$ and $\mb{T}_i =\emptyset.$ In particular, if both $\dl_j$ and $\dl_k$ are zero, then $q^{\ep_1}$ is $x$-separated. In all the other cases, it is not. Because of this we define a new path $ \qt(t)=(\qt_i(t))_{i \in \I}, t \in [0, T]$, which is $x$-separated, by the following deformation: 
$$ \qt_j(t) = (\tilde{x}_j, \tilde{y}_j, \tilde{z}_j)(t) =
 \begin{cases}
 (-x_j^{\ep_1}, y_j^{\ep_1}, z_j^{\ep_1})(t), & \text{ if } t \in \mb{T}_j\\
 (x_j^{\ep_1}, y_j^{\ep_1}, z_j^{\ep_1})(t), & \text{ if } t \in [0, t_j] \setminus \mb{T}_j \\
 (x_j^{\ep_1}, y_j^{\ep_1}, z_j^{\ep_1})(t)- 2\dl_j \e_1, & \text{ if } t \in [t_j, T],
 \end{cases}
 $$
 $$
 \qt_k(t) = (\tilde{x}_k, \tilde{y}_k, \tilde{z}_k)(t)=
 \begin{cases}
 (-x_k^{\ep_1}, y_k^{\ep_1}, z_k^{\ep_1})(t), & \text{ if } t \in \mb{T}_k\\
 (x_k^{\ep_1}, y_k^{\ep_1}, z_k^{\ep_1})(t), & \text{ if } t \in [0, t_k] \setminus \mb{T}_k \\
 (x_k^{\ep_1}, y_k^{\ep_1}, z_k^{\ep_1}) + 2\dl_k \e_1, & \text{ if } t \in [t_k, T],
 \end{cases}
 $$ 
$$\qt_i(t) = 
\begin{cases}
\qey_i(t) - 2 \dl_j \e_1, & \;\;  \text{ if } i \in \mathbf{I}_0,\\
\qey_i(t) + 2 \dl_k \e_1, & \;\;\text{ if } i \in \mathbf{I}_1,
\end{cases}
\;\; \forall t \in [0, T]. 
$$

The above deformation keeps the kinetic energy unchanged, i.e.,
 $$ \int_0^T K_{\I}(\dot{\qt}) \,dt = \int_0^T K_{\I}(\dot{q}^{\ep_1}) \,dt. $$
Meanwhile it does not increase the potential energy, as
$$ |\qt_{i_1}(t) - \qt_{i_2}(t)| \ge |q^{\ep_1}_{i_1}(t) - \qey_{i_2}(t)|, \;\; \forall t \in [0, T], \; \forall \{i_1 \ne i_2 \} \subset \I. $$
Therefore
$$ \A_{\I, T}(\qt) \le \A_{\I, T}(\qey) < \A_{\I, T}(q). $$

However $\qt(t)$ is not the path we are looking for, as it may not satisfies \eqref{eqn:xjk monotone}. Nevertheless using $\qt(t)$, we can define a new path $\qh(t)=(\qh_i(t))_{i \in \I}$, $t \in [0, T]$, as follows: 
$$ 
\begin{cases}
& \hat{x}_j(t)= - \int_0^t |\dot{\hat{x}}_j(s)| \,ds, \;\; \hat{x}_k(t)= \int_0^t |\dot{\hat{x}}_k(t)| \,dt, \\
& \hat{y}_i(t)= \tilde{y}_i(t), \;\; \hat{z}_i(t)= \tilde{z}_i(t), \; \text{ if } i \in \{j, k\}, 
\end{cases} \;\; \forall t \in [0, T],
$$
$$ \qh_i(t) = \begin{cases}
\qt_i(t) + (\hat{x}_j(T)-\tilde{x}_j(T))\e_1, \; &\text{ if } i \in \I_0, \\
\qt_i(t) + (\hat{x}_k(T)-\tilde{x}_k(T))\e_1, \; &\text{ if } i \in \I_1. 
\end{cases} \;\; \forall t \in [0, T].$$
Since $\qt$ is $x$-separated, so is $\qh$. Moreover $\qh$ also satisfies \eqref{eqn:xjk monotone}. Meanwhile $|\dot{\hat{x}}_i(t)|=|\dot{\tilde{x}}_i(t)|$ and $|\dot{\hat{y}}_i(t)|=|\dot{\tilde{y}}_i(t)|$, for any $t \in [0, T]$ and $i \in \{j, k\}$, it is easy to see
$$ \int_0^T K_{\I}(\dot{\qh}) \,dt= \int_0^T K_{\I}(\dot{\qt}) \,dt.$$
The above definition also implies 
$$   |\qh_{i_1}(t) - \qh_{i_2}(t)| \ge |\qt_{i_1}(t) - \qt_{i_2}(t)|, \;\; \forall t \in [0, T], \; \forall \{i_1 \ne i_2 \} \subset \I. $$
As a result, 
$$ \A_{\I, T}(\qh) \le \A_{\I, T}(\qt) \le  \A_{\I, T}(\qey) < \A_{\I, T}(q). $$
Therefore $\qh(t)$ is the $\qt^{\ep_1}(t)$ we are looking for. 
\end{proof}

Now we introduce the last deformation lemma of this section, which covers the remaining cases that can not be handled by the previous lemmas. 

\begin{lm}
 \label{lm:dfm2} If $(q_i)_{i \in \I}$ is $x$-separated and satisfies $x_j(T) < x_k(T)$, when $\phi_j^+=0$ or $\pi$, for $\ep_2>0$ small enough, $\qee: [0,T] \to \rr^{3N}$ defined as below:
 $$  \qee_j(t)=\begin{cases}
q_j(t) - t(2\ep_2 -t) \e_1, \; & \forall t \in [0, \ep_2], \\
q_j(t) -\ep_2^2 \e_1, \; & \forall t\in [\ep_2, T];
\end{cases} $$
$$ \qee_k(t) = \begin{cases}
 q_k(t) +t(2\ep_2 -t) \e_1, \; & \forall t \in [0, \ep_2],\\
 q_k(t) + \ep_2^2 \e_1, \; & \forall t\in [\ep_2, T], 
\end{cases} $$
 \begin{equation*} \qee_i(t)=
  \begin{cases}
    q_i(t) - \ep_2^2 \e_1, \; &  \text{ if } i \in \mathbf{I}_0,\\
    q_i(t) +\ep_2^2 \e_1, \; & \text{ if } i \in \mathbf{I}_1,    
  \end{cases} \; \forall t \in [0,T];
 \end{equation*}
 satisfies $\A_{\I,T}(\qee) < \A_{\I,T}(q)$.

 If $(q_i)_{i \in \I}$ is $y$-separated and satisfies $y_j(T) < y_k(T)$, the above result also holds if we replace $\e_1$ by $\e_2.$
\end{lm}
\begin{rem}
 We draw the readers attention to the additional conditions $x_j(T)< x_k(T)$ and $y_j(T)< y_k(T)$ as they are crucial in our proof. 
\end{rem}

\begin{proof}
 We give the proof for $(q_i)_{i \in \I}$ being $x$-separated, the other case is similar. 

 By the definition of $\qee$, a straight forward computation shows
 $$ \int_0^T K_{\I}(\qd^{\ep_2}) - K_{\I}(\qd) \,dt >0, \;\; \int_0^T U_{\I}(\qee) - U_{\I}(q) \,dt <0. $$
 To get the desired result, we need to improve the above estimates. 
 
 The key to control the increment of the kinetic energy is to get estimates of $\xd_i(t)$, $i \in \{j, k \}$, as $t$ goes to zero. Recall that $\dot{x}_i(t) = \xfd_i(t) + \dot{x}_c(t).$ By \eqref{eqn:sph co}, 
 $$ \xfd_i = \dot{r}_i \sin \phi_i \cos \tht_i + r_i \dot{\phi}_i \cos \phi_i \cos \tht_i - r_i \thd_i \sin \phi_i \sin \tht_i, \;\; \forall i \in \{j, k\}. $$
 By Proposition \ref{prop:angle}, $\phi_k^+=\pi$, if  $\phi_j^+=0$; $\phi_k^+=0$, if  $\phi_i^+=\pi$. Then Proposition \ref{prop:sundman} and \ref{prop:angle} imply
 \begin{equation}
  \label{eqn:xfdjk} \xfd_i(t) \le Ct^{\frac{2}{3}}, \text{ as } t \to 0^+, \;\; \forall i \in \{j, k\}. 
 \end{equation}
 Although $m_j$ and $m_k$ collide at $t=0$, their center of mass, $q_c(t)$ is still $C^2$ at $t=0$ and we claim 
  $$\dot{x}_c(0) = 0.$$
 By a contradiction argument, assume $\dot{x}_c(0) >0$, then by \eqref{eqn:xfdjk}
 $$ \xd_j(t) = \xfd_j(t) + \dot{x}_c(t) >0, \text{ for } t>0 \text{ small enough.}$$
 This means 
 $$ x_j(t) >x_j(0), \text{ for } t>0 \text{ small enough}. $$
 As $(q_i)_{i \in \I}$ is $x$-separated, this is a contradiction to condition (i) in Definition \ref{dfn:sepa}. Similarly if we assume $\dot{x}_c(0)<0$, then 
 $$  x_k(t) < x_k(0), \text{ for } t>0 \text{ small enough}, $$
 which is again a contradiction to condition (i) in Definition \ref{dfn:sepa}. 
 This proves our claim. As a result
 $$ |\dot{x}_c(t)| \le Ct, \text{ as } t \to 0^+. $$
 Combine this with \eqref{eqn:xfdjk}, one gets
 \begin{equation}
 \label{eqn:xdjk} |\dot{x}_i(t)| \le C_1t^{\frac{2}{3}}, \text{ as } t \to 0^+, \; \forall i \in \{j, k\}. 
 \end{equation}
 Based on the above estimates and the definition of $\qee$, 
 \begin{align*}
 \int_0^T K_{\I}(\dot{q}^{\ep_2}) &- K_{\I}(\qd) \, dt  = \ey \int_0^T m_j(|\dot{q}^{\ep_2}_j|^2 - |\qd_j|^2) + m_k(|\dot{q}^{\ep_2}_k|^2 - |\qd_k|^2) \,dt \\
 & = 2 \int_0^{\ep_2} m_j((\ep_2-t)^2 - \xd_j(t)(\ep_2-t)) +m_k((\ep_2-t)^2 + \xd_k(t)(\ep_2 -t)) \, dt \\
 & \le 4(m_j+m_k) \int_0^{\ep_2} (\ep_2-t)^2 + C_1t^{\frac{2}{3}}(\ep_2-t) \, dt \le C_2 \ep_2^{\frac{8}{3}}.
 \end{align*}

Now we estimate the change in potential energy. As $x_k(T) -x_j(T)>0$, there exist a $\dl >0$ small enough ($T-\dl> \ep_2$ will be enough), and $C_3, C_4$, such that 
 \begin{equation}
 \label{eqn:xk-xj} \forall t \in [T-\dl, T], \;\; \begin{cases}
 x_k(t) - x_j(t) & \ge C_3, \\
 |q_j(t) - q_k(t)|^{-1} & \ge C_4.
 \end{cases}
 \end{equation}
 By the definition of $\qee$, for any  $t \in [\ep_2, T]$,
 $$ |\qee_k(t) - \qee_j(t)|^{-1} = [|q_k(t)-q_j(t)|^2 + 4 \ep_2^2(x_k(t)-x_j(t))+ 4\ep_2^4]^{-\ey}.$$
 Combine this with \eqref{eqn:xk-xj}, we get that for any $t \in [T-\dl, T]$,
\begin{align*}
|\qee_k(t) & -\qee_j(t)|^{-1} -|q_k(t) -q_j(t)|^{-1}\\ 
&\le -2|q_k(t)-q_j(t)|^{-1}[(x_k(t)-x_j(t))\ep_2^2 + \ep_2^4] + o(\ep_2^2) \le -C_5 \ep_2^2.
\end{align*}
This means
\begin{equation}
 \int^T_0 U_{\I}(\qee)- U_{\I}(q)\,dt \le \int_{T-\dl}^T \frac{m_j m_k}{|\qee_k-\qee_j|}- \frac{m_j m_k}{|q_k-q_j|} \,dt \le -C_6 \dl \ep_2^2. 
 \end{equation} 

As a result,  for $\ep_2$ small enough
$$ \A_{\I,T}(\qee) -\A_{\I,T}(q) \le -C_6 \dl \ep_2^2 + C_2 \ep_2^{\frac{8}{3}} < 0.$$
\end{proof}

So far in this section, we have only discussed the case when an isolated binary collision is happening at $t=0$. If such an isolated binary collision occurs at $t=T$, by simply reversing the time we can get similar results stated as above.

\section{Proof of Theorem \ref{thm:main}} \label{sec:proof}
This section will be devoted to the proof of Theorem \ref{thm:main}. Given an arbitrary $\om \in \Om_{[n/2]}$, since $\hlm_{<, \om}$ is not a closed set, in order to apply the direct method of calculus of variation, we consider its weak closure in $\Lmd^n$, which will be denoted by $\lmn_{\le, \om}$. Alternatively we may also define $\lmn_{\le, \om}$ as follows.
\begin{dfn} \label{dfn:lmn} $\lmn_{\le, \om}$ is a subset of $\lmn$ with every loop $q$ satisfies the following conditions:
\begin{enumerate}
\item[(i).] $q$ is \textbf{$x$-monotone}: $ x_0(t_1) \le x_0(t_2), \;\; \forall 0 \le t_1 < t_2 \le \frac{n}{4};$
\item[(ii).] $q$ is \textbf{$y$-monotone}: $ y_0(t_1) \le y_0(t_2), \;\; \forall 0 \le t_1 < t_2 \le \frac{n}{4};$
\item[(iii).] $q$ satisfies \textbf{$\om$-topological constraints} given by condition \eqref{eqn:tc}.
\end{enumerate}
\end{dfn}
\begin{rem} \label{rem:monotone} $(a).$ the conditions required by $x/y$-monotone in the above definition is weaker than those required by  strictly $x/y$-monotone in definition \ref{dfn:smc};\\
$(b)$. for any $i \in \{0, \dots, [n/2]\}$, if $q \in \hat{\Lmd}^n_{\le, \om}$, then $z_0(i/2) \ne 0$, and if $q \in \lmn_{\le, \om}$, $z_0(i/2)$ could be zero.
  
\end{rem}

\begin{prop}
\label{prop:coer} For any $\om \in \omn$, there exists a $\qo \in \lmn_{\le, \om}$ satisfying 
$$ \A_n(\qo) = \inf\{\A_n(q)|\; q \in \lmn_{\le, \om} \}.$$
\end{prop}
\begin{proof}
Since $\lmn_{\le, \om}$ is weakly closed and $\A_n$ is weakly lower semi-continuous with respect to the Sobolev norm $H^1$, by a standard result of calculus of variation, it is enough to show $\A_n$ is coercive in $\lmn_{\le, \om}:$ $\A_n(q^k) \to \infty$, if $\|q^k\|_{H^1} \to \infty$, as $k$ goes to infinity, where $\{ q^k \in \lmn_{\le, \om} \}_{k=0}^{\infty}.$  

By Definition \ref{dfn:tc}, for any $q^k$, there is a $t_0 \in [0, n)$, such that $z_0(t_0)=0$. Then
$$ |z_0^k(s)| = |z_0^k(s) -z^k_0(t_0)| \le \int_0^{n} |\dot{z}^k_{0}(t)| \,dt, \;\; \forall s \in [0, n). $$
By Cauchy-Schwartz inequality,
$$ |z^k_0(s)|^2 \le \Big( \int_0^n |\dot{z}^k_0(t)| \,dt \Big)^2 \le n \int_0^n |\dot{z}^k_0(t)|^2 \,dt, \;\; \forall s \in [0, n). $$ 
Then 
$$ \int_0^n |z^k_0(t)|^2 \,dt \le n^2 \int_0^n |\dot{z}^k_0(t)|^2 \,dt. $$

By \eqref{eqn:q0symmetry}, $y^k_0(0) = x^k_0(n/4)=0$, similar computations as above imply
$$ \int_0^n |x^k_0(t)|^2 \,dt \le n^2 \int_0^n |\dot{x}^k_0(t)|^2 \,dt, \quad \quad \int_0^n |y^k_0(t)|^2 \,dt \le n^2 \int_0^n |\dot{y}^k_0(t)|^2 \,dt. $$
As a result, 
\begin{equation}
\label{eqn:q0 q0dot} \int_0^n |q^k_0(t)|^2 \,dt \le n^2 \int_0^n |\dot{q}^k_0(t)|^2 \,dt.
\end{equation}
Due to the symmetric constraints,
\begin{equation}
\label{eqn:q q0} \|q^k\|_{H^1}^2 = \int_0^n \sum_{i \in \N} \big(|q^k_i|^2 + |\dot{q}^k_i|^2 \big) \,dt= 2n \int_0^n |q_0^k|^2 + |\dot{q}_0^k|^2 \,dt,
\end{equation}
\begin{equation}
\label{eqn:A > K} \A_n(q^k) \ge \ey \int_0^n \sum_{i \in \N} |\dot{q}_i^k|^2 \,dt = n \int_0^n |\dot{q}_0^k|^2 \,dt.
\end{equation}
Combining \eqref{eqn:q0 q0dot}, \eqref{eqn:q q0} and \eqref{eqn:A > K}, we get 
$$ \A_n(q^k) \ge \frac{1}{2(n^2+1)} \| q^k\|^2_{H^1}. $$
This finishes the proof.
\end{proof}

Let $\qo$ be a minimizer of the action functional $\A_n$ in $\lmn_{\le, \om}$, whose existence has been established by Proposition \ref{prop:coer}.

\begin{prop}
\label{prop:smc} If $x^{\om}_0(0)<0$, then $\qo$ is strictly $x$-monotone; if $y_0^{\om}(n/4) >0$, then  $\qo$ is strictly $y$-monotone.  
\end{prop}

\begin{proof}
We give the details for strictly $x$-monotone, while the other is similar. For simplicity, let $q=\qo$. 

First let's assume $n$ is even ($n=2\ell$). Due to the symmetric constraints, $q$ is strictly $x$-monotone if and only if
$$ x_i(t_1) < x_i(t_2), \; \forall 0 \le t_1 < t_2 \le 1/2, \; \forall i \in \{0, \dots, n/2 -1\}. $$
Since $q$ is $x$-monotone, we already have 
$$ x_i(t_1) \le x_i(t_2), \; \forall 0 \le t_1 < t_2 \le 1/2, \; \forall i \in \{0, \dots, n/2 -1\}. $$
By a contradiction argument, let's assume there exist $0 \le t_1 < t_2 \le 1/2$ and $k \in \{0, \dots, n/2-1\}$, such that 
\begin{equation} \label{eqn:xt=xt1}
x_k(t) \equiv x_k(t_1), \;\; \forall t \in [t_1, t_2]. 
\end{equation}
By the symmetric constraints (the action of $h_1$ and $h_2$),
\begin{equation}
\label{eqn:ksym} q_{k+n}(t) = \R_x q_{k}(t), \; q_{k +\en}(t) = \R_z q_{k}(t), \; q_{k +\frac{3}{2}n}(t) = \R_x q_{k+\frac{n}{2}}(t), \;\; \forall t.
\end{equation}
Together with \eqref{eqn:xt=xt1}, they imply 
\begin{equation}
\label{eqn:xit=xit1} x_i(t) \equiv x_i(t_1), \; \forall t \in [t_1, t_2], \; \forall i \in \I_0,
\end{equation}
where $\I_0 = \{ k +\frac{i}{2}n|\; i =0, \dots, 3 \}.$ 

In the following we will find a new path $\qe \in \lmn_{\le, \om}$ with $\A_{1/2}(\qe) < \A_{1/2}(q)$, for $\ep>0$ small enough. This gives us a contradiction. To achieve this, we need some information regarding the relative positions of the masses. Since
$$ x_{n/2}(0) -x_0(0) = x_0(n/2)-x_0(0) = -2 x_0(0) > 0,$$
there exist $t_0, \dl_1, \dl_2 >0$ small enough (independent of $\ep$), such that
\begin{align}
 \label{eqn:dl1} x_{n/2}(t) - x_0(t) \ge \dl_1, \;\; & \forall t \in [0, t_0], \\
 \label{eqn:dl2} |q_{n/2}(t)-q_0(t)|^{-1} \ge \dl_2, \;\; & \forall t \in [0, t_0].
\end{align}

Depending on the value of $k$, two different cases need to be considered. 

\emph{Case 1}: $k \in \{0, \dots, [(n-2)/4] \}$. Choose an arbitrary $\qt \in \lmn_{<, \om}$. Based the strictly monotone property of $\qt$, the following two subsets of indices are well-defined and independent of the choice of $\qt$, 
\begin{align*}
 \I_1^- &= \{ i \in \N | \; \xt_i(t) < \xt_k(0), \;\; \forall t \in (0, 1/2) \};\\
 \I_1^+ & = \{ i \in \N | \; \xt_i(t) > \xt_{k+\en}(0), \;\; \forall t \in (0, 1/2) \}.  
\end{align*}
Now we define a new path $\qe \in \lmn_{\le, \om}$ as follows: 
\begin{align*}
\qe_i(t) & = \begin{cases}
               q_i(t) - \ep \e_1, \;\; & \forall t \in [0, t_1], \\
               q_i(t) + \frac{t-t_2}{t_2-t_1} \ep \e_1, \;\; &\forall t \in [t_1, t_2], \\
               q_i(t), \;\; & \forall t \in [t_2, 1/2],
              \end{cases}
              \;\; && \text{ if } i \in \{k, k +n \}; \\ 
\qe_i(t) & = \begin{cases}
              q_i(t) +\ep \e_1, \;\; & \forall t \in [0, t_1], \\
              q_i(t) + \frac{t_2-t}{t_2-t_1}\ep \e_1, \; \; &\forall t \in [t_1, t_2], \\
              q_i(t), \;\; & \forall t \in [t_2, 1/2],
             \end{cases}
             \;\; && \text{ if } i \in \{ k+\en, k +\frac{3}{2}n \};\\
\qe_i(t) & = \begin{cases}
 			  q_i(t) -\ep \e_1, \;\; & \text{ if } i \in \I_1^-, \\
              q_i(t) +\ep \e_1, \;\; & \text{ if } i \in \I_1^+, \\
 			  q_i(t),  \;\; & \text{ if } i \in \N \setminus (\I_0 \cup \I_1^- \cup \I_1^+),            
            \end{cases}	
            \;\;  && \;\;\forall t \in [0, 1/2].             
\end{align*}
By the definition of $\qe$ and \eqref{eqn:xit=xit1}, 
\begin{equation}
 \label{eqn:K} \begin{split} \int_{0}^{\ey} K(\qd^{\ep}(t)) - K(\qd(t)) \,dt & = \ey \int_{t_1}^{t_2} \sum_{i \in \I_0} ( |\qd^{\ep}_i(t)|^2 - |\qd_i(t)|^2) \,dt \\
                & = 2 \int_{t_1}^{t_2} \frac{\ep^2}{(t_2 -t_1)^2} \,dt = \frac{2}{t_2-t_1}\ep^2.
               \end{split}
\end{equation}

To estimate the change in potential energy, notice that
\begin{equation}
 \label{eqn:qe ge q} |\qe_i(t)-\qe_j(t)| \ge |q_i(t)-q_j(t)|, \;\; \forall t \in [0,1/2], \; \forall \{i \ne j \} \subset \N. 
\end{equation}
When $k >0$, $0 \in \I_1^-$ and $n/2 \in \I_1^+$. By the definition of $\qe$, for any $t \in [0, t_0]$,
\begin{align*}
 |\qe_{\en}(t)- \qe_0(t)|^2 & = (x_{\en}(t)-x_0(t) + 2\ep)^2 + (y_{\en}(t) - y_0(t))^2 + (z_{\en}(t) -z_0(t))^2 \\ 
 & = |q_{\en}(t)- q_0(t)|^2 + 4(x_{\en}(t)-x_0(t))\ep +4\ep^2 \\
 & \ge |q_{\en}(t)- q_0(t)|^2 + 4\dl_1 \ep +4\ep^2,
\end{align*}
where the last inequality follows from \eqref{eqn:dl1}. Combine this with \eqref{eqn:dl2}, we get
\begin{align*}
 |\qe_{\en}(t)&-\qe_0(t)|^{-1} - |q_{\en}(t) -q_0(t)|^{-1}\\
 & \le \frac{1}{|q_{\en}(t)-q_0(t)|} \Big[ \big(1 + \frac{4\dl_1 \ep}{|q_{\en}(t)-q_0(t)|^2}+\frac{4\ep^2}{|q_{\en}(t)-q_0(t)|^2} \big)^{-\ey} -1 \Big] \\
 & \le -\frac{2\dl_1 \ep}{|q_{\en}(t)-q_0(t)|^3}+o(\ep) \le -2 \dl_1 \dl_2^3 \ep + o(\ep). 
\end{align*}
Meanwhile \eqref{eqn:qe ge q} implies 
\begin{equation*}
 \int_{0}^{\ey} U(\qe(t)) - U(q(t)) \,dt \le \int_{0}^{t_0} |\qe_{n/2}(t)-\qe_0(t)|^{-1} - |q_{n/2}(t) - q_0(t)|^{-1} \,dt. 
\end{equation*}
As a result,
\begin{equation}
 \label{eqn:U1} \int_0^{\ey}U(\qe)- U(q)\,dt \le \int_0^{t_0} -2\dl_1 \dl_2^3 \ep +o(\ep)\,dt = -C_1(\dl_1, \dl_2, t_0)\ep + o(\ep). 
\end{equation}
When $k=0$, $\I_1^-= \I_1^+ = \emptyset.$ However if $t_0 \le t_1$, all the above estimates still hold for any $t \in [0, t_0]$, and so is \eqref{eqn:U1}. On the other hand, if $t_0 >t_1$, then for any $t \in [t_1, t_3]$ ($t_3= \min\{t_0, t_2\})$, 
\begin{align*}
 |\qe_{\en}(t)- \qe_0(t)|^2 & = (x_{\en}(t)-x_0(t) + 2\frac{t-t_1}{t_2-t_1}\ep)^2 + (y_{\en}(t) - y_0(t))^2 + (z_{\en}(t) -z_0(t))^2 \\ 
 & = |q_{\en}(t)- q_0(t)|^2 + 4(x_{\en}(t)-x_0(t))\frac{t-t_1}{t_2-t_1}\ep +4(\frac{t-t_1}{t_2-t_1})^2 \ep^2 \\
 & \ge |q_{\en}(t)- q_0(t)|^2 + 4\dl_1\ep  \frac{t-t_1}{t_2-t_1}+o(\ep).
\end{align*}
Combining this with \eqref{eqn:dl2}, we get
\begin{align*}
 |\qe_{\en}(t)&-\qe_0(t)|^{-1} - |q_{\en}(t) -q_0(t)|^{-1}\\
 & \le \frac{1}{|q_{\en}(t)-q_0(t)|} \Big[ \big(1 + \frac{4\dl_1 \ep \frac{t-t_1}{t_2-t_1}}{|q_{\en}(t)-q_0(t)|^2}+o(\ep) \big)^{-\ey} -1 \Big] \\
 & \le -\frac{2\dl_1 \ep  \frac{t-t_1}{t_2-t_1}}{|q_{\en}(t)-q_0(t)|^3}+o(\ep) \le -2 \dl_1 \dl_2^3 \ep \frac{t-t_1}{t_2-t_1} + o(\ep). 
\end{align*}
By \eqref{eqn:qe ge q},
\begin{equation}
 \label{eqn:U2} \begin{split}
 \int_0^{\ey} U(\qe(t)) - U(q(t)) \,dt & \le \int_{t_1}^{t_3} |\qe_{n/2}(t)-\qe_0(t)|^{-1} - |q_{n/2}(t) - q_0(t)|^{-1} \,dt, \\
                 & \le \int_{t_1}^{t_3} -2 \dl_1 \dl_2^3 \ep \frac{t-t_1}{t_2-t_1} + o(\ep) \,dt \\
                  & \le -C_2(\dl_1, \dl_2, t_0, t_1, t_2) \ep +o(\ep). 
                \end{split}
\end{equation}
By the above estimates, in particular \eqref{eqn:K}, \eqref{eqn:U1} and \eqref{eqn:U2}, we get 
$$ \A_{1/2}(\qe) - \A_{1/2}(q) \le -C_3(\dl_1, \dl_2, t_0, t_1, t_2)\ep + o(\ep) < 0$$
for $\ep>0$ small enough. This contradicts the fact that $q$ is an action minimizer in $\lmn_{\om, \le}$ and finishes our proof of \emph{Case 1}.   

\emph{Case 2}: $k \in \{[(n+2)/4], \dots, n/2-1 \}$. Let
\begin{align*}
 \I_2^- &= \{ i \in \N | \; \xt_i(t) < \xt_{k+\en}(1/2), \;\; \forall t \in (0, 1/2) \};\\
 \I_2^+ & = \{ i \in \N | \; \xt_i(t) > \xt_{k}(1/2), \;\; \forall t \in (0, 1/2) \}, 
\end{align*}
and $\qe \in \lmn_{\le, \om}$ be defined as follows:
\begin{align*}
\qe_i(t) & = \begin{cases}
               q_i(t), \;\; & \forall t \in [0, t_1], \\
               q_i(t) + \frac{t-t_1}{t_2-t_1} \ep \e_1, \;\; &\forall t \in [t_1, t_2], \\
               q_i(t)+\ep \e_1, \;\; & \forall t \in [t_2, 1/2],
              \end{cases}
              \;\; && \text{ if } i \in \{k, k +n \}; \\ 
\qe_i(t) & = \begin{cases}
              q_i(t), \;\; & \forall t \in [0, t_1], \\
              q_i(t) + \frac{t_1 -t}{t_2-t_1}\ep \e_1, \; \; &\forall t \in [t_1, t_2], \\
              q_i(t)-\ep \e_1, \;\; & \forall t \in [t_2, 1/2],
             \end{cases}
             \;\; && \text{ if } i \in \{ k+\en, k +\frac{3}{2}n \};\\
\qe_i(t) & = \begin{cases}
 			  q_i(t) -\ep \e_1, \;\; & \text{ if } i \in \I_2^-, \\
              q_i(t) +\ep \e_1, \;\; & \text{ if } i \in \I_2^+, \\
 			  q_i(t),  \;\; & \text{ if } i \in \N \setminus (\I_0 \cup \I_2^- \cup \I_2^+),            
            \end{cases}	
            \;\;  && \;\;\forall t \in [0, 1/2].             
\end{align*}

By the same argument given in \emph{Case 1}, $\A_{1/2}(\qe) < \A_{1/2}(q)$, for $\ep >0$ small enough, which is absurd. This finishes our proof for even $n$. 

Now let's assume $n$ is odd ($n = 2\ell+1$). Because of the symmetric constraints, $[0, 1/4]$ is a fundamental domain and $q$ is strictly $x$-monotone if and only if 
$$ x_i(t_1) < x_i(t_2), \; \forall 0 \le t_1 < t_2 \le 1/4, \; \forall i \in \{0, \dots, n-1\}.$$ 
Like before as $q$ is $x$-monotone, we already have 
$$ x_i(t_1) \le x_i(t_2), \; \forall 0 \le t_1 < t_2 \le 1/4, \; \forall i \in \{0, \dots, n-1\}.$$ 
Again by a contradiction argument, let's assume there exist $0 \le t_1 < t_2 \le 1/4$ and $k \in \{0, \dots, n-1\}$, satisfying \eqref{eqn:xt=xt1}. Then by the action of $h_1$
\begin{equation}
\label{eqn:ksym odd} q_{k+n}(t) = \R_x q_k(t), \;\; \forall t.
\end{equation}
If we let $\I_0= \{k, k +n\}$ for the rest of the proof, then
\begin{equation}
\label{eqn:xit=xit1 odd} x_i(t) \equiv x_i(t_1), \; \forall t \in [t_1, t_2], \; \forall i \in \I_0. 
\end{equation}
Just like above, we will find a new path $\qt \in \lmn_{\le, \om}$ with $\A_{1/4}(\qe) < \A_{1/4}(q)$, for $\ep >0$ small enough and reach a contradiction. The relative positions of the masses will still be needed to estimate the change in potential energy. Since 
$$ x_0(0) < 0, \quad x_{[n/2]}(0) \ge x_0(n/4)=0,$$
there exist $t_0, \dl_3, \dl_4 >0$ small enough (independent of $\ep$), such that 
\begin{align}
\label{eqn:dl3} x_{[n/2]}(t) - x_0(t) \ge \dl_3, \;\; & t \in [0, t_0], \\
\label{eqn:dl4} |q_{[n/2]}(t) - q_0(t)|^{-1} \ge \dl_4, \;\; & t \in [0, t_0].
\end{align}

Depending on the value of $k$, four different cases will be considered (to distinguish from the previous two, we will count from 3).  

\emph{Case 3}: $k \in \{0, \dots, [n/4] \}$. Let $\qt$ be the same as above, we set
\begin{align*}
 \I_3^- &= \{ i \in \N | \; \xt_i(t) < \xt_k(0), \;\; \forall t \in (0, 1/4) \};\\
 \I_3^+ & = \{ i \in \N | \; \xt_i(t) > \xt_{[\en]-k}(1/4), \;\; \forall t \in (0, 1/4) \}, 
\end{align*}
Now define a $\qe \in \lmn_{\le, \om}$ as follows:
\begin{align*}
\qe_i(t) & = \begin{cases}
               q_i(t) - \ep \e_1, \;\; & \forall t \in [0, t_1], \\
               q_i(t) + \frac{t-t_2}{t_2-t_1} \ep \e_1, \;\; &\forall t \in [t_1, t_2], \\
               q_i(t), \;\; & \forall t \in [t_2, 1/4],
              \end{cases}
              \;\; && \text{ if } i \in \{k, k +n \}; \\ 
\qe_i(t) & = \begin{cases}
 			  q_i(t) -\ep \e_1, \;\; & \text{ if } i \in \I_3^-, \\
              q_i(t) +\ep \e_1, \;\; & \text{ if } i \in \I_3^+, \\
 			  q_i(t),  \;\; & \text{ if } i \in \N \setminus (\I_0 \cup \I_3^- \cup \I_3^+),            
            \end{cases}	
            \;\;  && \;\;\forall t \in [0, 1/4].             
\end{align*}
By the above definition of $\qe$ and \eqref{eqn:xit=xit1 odd}, 
\begin{equation}
\label{eqn:K odd} \begin{split} \int_{0}^{\frac{1}{4}} K(\qd^{\ep}(t)) - K(\qd(t)) \,dt & = \ey \int_{t_1}^{t_2} \sum_{i \in \I_0} ( |\qd^{\ep}_i(t)|^2 - |\qd_i(t)|^2) \,dt \\
                & = \int_{t_1}^{t_2} \frac{\ep^2}{(t_2 -t_1)^2} \,dt = \frac{\ep^2}{t_2-t_1}.
               \end{split}
\end{equation}
Meanwhile \eqref{eqn:qe ge q} still holds after replacing $[0, 1/2]$ by $[0,1/4]$. With \eqref{eqn:dl3} and \eqref{eqn:dl4}, by computations similar to those given in \emph{Case 1}, we get
\begin{equation}
\label{eqn:U odd} \int_0^{\frac{1}{4}} U(\qe) - U(q)\,dt \le -C_4(\dl_3, \dl_4, t_0, t_1, t_2) \ep + o(\ep).
\end{equation}
By \eqref{eqn:K odd} and \eqref{eqn:U odd}, $\A_{1/4}(\qe) < \A_{1/4}(q)$, for $\ep>0$ small enough, which is absurd. This finishes our proof of \emph{Case 3}. 

The proofs of the remaining cases are similar. We just give the definition of $\qe \in \lmn_{\le, \om}$ in each one and omit details.  

\emph{Case 4}: $k \in \{[n/4], \dots, [(n-1)/2] \}.$ Let
\begin{align*}
 \I_4^- &= \{ i \in \N | \; \xt_i(t) < \xt_{[\en]-k}(0), \;\; \forall t \in (0, 1/4) \};\\
 \I_4^+ & = \{ i \in \N | \; \xt_i(t) > \xt_{k}(1/4), \;\; \forall t \in (0, 1/4) \},  
\end{align*}
and
\begin{align*}
\qe_i(t) & = \begin{cases}
               q_i(t), \;\; & \forall t \in [0, t_1], \\
               q_i(t) + \frac{t-t_1}{t_2-t_1} \ep \e_1, \;\; &\forall t \in [t_1, t_2], \\
               q_i(t)+\ep \e_1, \;\; & \forall t \in [t_2, 1/4],
              \end{cases}
              \;\; && \text{ if } i \in \{k, k +n \}; \\ 
\qe_i(t) & = \begin{cases}
 			  q_i(t) -\ep \e_1, \;\; & \text{ if } i \in \I_4^-, \\
              q_i(t) +\ep \e_1, \;\; & \text{ if } i \in \I_4^+, \\
 			  q_i(t),  \;\; & \text{ if } i \in \N \setminus (\I_0 \cup \I_4^- \cup \I_4^+),            
            \end{cases}	
            \;\;  && \;\;\forall t \in [0, 1/4].             
\end{align*}
\emph{Case 5}: $k \in \{[(n+1)/2], \dots, [3n/4] \}$. Let
\begin{align*}
 \I_5^- &= \{ i \in \N | \; \xt_i(t) < \xt_{[\frac{3}{2}n]-k}(1/4), \;\; \forall t \in (0, 1/4) \};\\
 \I_5^+ & = \{ i \in \N | \; \xt_i(t) > \xt_{k}(0), \;\; \forall t \in (0, 1/4) \},  
\end{align*}
and
\begin{align*}
\qe_i(t) & = \begin{cases}
               q_i(t) + \ep \e_1, \;\; & \forall t \in [0, t_1], \\
               q_i(t) + \frac{t_2-t}{t_2-t_1} \ep \e_1, \;\; &\forall t \in [t_1, t_2], \\
               q_i(t), \;\; & \forall t \in [t_2, 1/4],
              \end{cases}
              \;\; && \text{ if } i \in \{k, k +n \}; \\ 
\qe_i(t) & = \begin{cases}
 			  q_i(t) -\ep \e_1, \;\; & \text{ if } i \in \I_5^-, \\
              q_i(t) +\ep \e_1, \;\; & \text{ if } i \in \I_5^+, \\
 			  q_i(t),  \;\; & \text{ if } i \in \N \setminus (\I_0 \cup \I_5^- \cup \I_5^+),            
            \end{cases}	
            \;\;  && \;\;\forall t \in [0, 1/4].             
\end{align*}
\emph{Case 6}: $k \in \{ [3n/4]+1, \dots, n-1\}.$ Let
\begin{align*}
 \I_6^- &= \{ i \in \N | \; \xt_i(t) < \xt_k(1/4), \;\; \forall t \in (0, 1/4) \};\\
 \I_6^+ & = \{ i \in \N | \; \xt_i(t) > \xt_{[\frac{3}{2}n]-k}(0), \;\; \forall t \in (0, 1/4) \},  
\end{align*}
and
\begin{align*}
\qe_i(t) & = \begin{cases}
               q_i(t), \;\; & \forall t \in [0, t_1], \\
               q_i(t) + \frac{t_1-t}{t_2-t_1} \ep \e_1, \;\; &\forall t \in [t_1, t_2], \\
               q_i(t)-\ep \e_1, \;\; & \forall t \in [t_2, 1/4],
              \end{cases}
              \;\; && \text{ if } i \in \{k, k +n \}; \\ 
\qe_i(t) & = \begin{cases}
 			  q_i(t) -\ep \e_1, \;\; & \text{ if } i \in \I_6^-, \\
              q_i(t) +\ep \e_1, \;\; & \text{ if } i \in \I_6^+, \\
 			  q_i(t),  \;\; & \text{ if } i \in \N \setminus (\I_0 \cup \I_6^- \cup \I_6^+),            
            \end{cases}	
            \;\;  && \;\;\forall t \in [0, 1/4].             
\end{align*}
\end{proof}

The conditions required by Proposition \ref{prop:smc} will be established by the following lemma. We point out that condition \eqref{eqn:odd om} in Theorem \ref{thm:main} has not been used so far and it comes into the picture through the following lemma.

\begin{lm}
\label{lm:nondeg} If $\om$ satisfies \eqref{eqn:odd om}, then $x^{\om}_0(0) < 0 < y_0^{\om}(n/4).$
\end{lm}
This lemma shows $\qo$ can't be contained in the $yz$ or $xz$-plane. A proof of it will be given in Section \ref{sec:app}. Proposition \ref{prop:smc} and Lemma \ref{lm:nondeg} immediately imply 

\begin{cor}
\label{cor: strictly x y monotone}  If $\om$ satisfies \eqref{eqn:odd om}, then $\qo \in \lmn_{<, \om}$.
\end{cor}

By the above corollary, $\qo$ is strictly $x$ and $y$-monotone. However because of the monotone constraints, right now we can't say $\qo(t)$ satisfies equation \eqref{eqn:nbody}, even if it is collision-free. For that we need the following lemma.   
\begin{lm}
\label{lm: x y dot >0} If $\qo \in \lmn_{<, \om}$ is collision-free, then it satisfies \eqref{eq: x dot >0} and \eqref{eq: y dot >0} in Theorem \ref{thm:main}. 
\end{lm}

\begin{proof}
We give a proof of \eqref{eq: x dot >0}, while \eqref{eq: y dot >0} can be proven similarly. For simplicity let $q= \qo$. Because $q$ is an action minimizer of the action functional $\A$ in $\lmn_{\le, \om}$, $\xd_0(t)$ exists for any $t$. Otherwise the path of $m_0$ has a corner at $x_0(t)$, which prevents $q$ being a minimizer. 

Because of the minimizing property of $q$, the first variation of $\A$ vanishes at $q$ and this implies the path of $m_0$ must hit the $x$-axis orthogonally at $t=0$, which means $\xd_0(0)=0$. For the rest of \eqref{eq: x dot >0}, first by the monotone constraints, $\xd_0(t) \ge 0$, for any $t \in (0, n/4]$. By a monotone constraint, let's assume $\xd_0(t_0)=0$, for some $t_0 \in (0, n/4]$. Under such an assumption, we will show that, for $\ep>0$ small enough, there is another loop $\qe \in \lmn_{\le, \om}$ with $\A_{1/2}(\qe) < \A_{1/2}(q)$, which is a contradiction. 

First let's assume $t_0 \in (0, n/4)$. Recall that a loop $\qt \in \lmn_{\le, \om}$ is uniquely determined through symmetries, once $\qt_0(t), t \in [0, n/4]$ is defined. Hence for $\ep>0$ small enough, we define the loop $\qe \in \lmn_{\le, \om}$ by defining $\qe_0(t)$, $t \in [0, n/4]$ as follows:
\begin{equation*}
  \qe_0(t) = \begin{cases}
  q(t)- \ep^2 \e_1, \;\; & \forall t \in [0, t_0-\ep], \\
  q(t)+(t-t_0)(2\ep-|t-t_0|), \;\; & \forall t \in [t_0-\ep, t_0 +\ep], \\
  q(t)+\ep^2 \e_1, \;\; & \forall t \in [t_0+\ep, n/4]. 
  \end{cases}
  \end{equation*}  
Actually we need to further shift $\qe_0$ to the left by $\ep^2$ along the $x$-axis (we will still denote it by $\qe_0$). This is to ensure $x_0(n/4)=0$. Since $\xd_0(t_0)=0$, we have 
\begin{equation*}
|\xd_0(t)| \le C|t-t_0|, \; \text{ for } |t-t_0| \text{ small enough}. 
\end{equation*}
As a result, 
\begin{align*}
\int_0^{1/2} K(\qd^{\ep})- K(\qd) \,dt & = 4 \int_0^{n/4} \ey \big( |\qd^{\ep}_0|^2 - |\qd_0|^2 \big) \,dt = 2 \int_{											 t_0-\ep}^{t_0+\ep} |\qd_0^{\ep}|^2-|\qd_0|^2 \,dt \\								   						       & = 2  \int_{t_0-\ep}^{t_0+\ep} 4(\ep-|t-t_0|)^2+ 4 \xd_0(t)(\ep-|t-t_0|) \,dt \\
									   & \le C_1 \ep^3.
\end{align*}		
For the change in potential energy, by the definition of $\qe_0$, 
$$ |\qe_i(t)-\qe_j(t)| \ge |q_i(t)-q_j(t)|, \;\; \forall t \in [0, 1/2], \; \forall \{ i \ne j\} \subset \N. $$
Furthermore by similarly estimates given in the proof of Lemma \ref{lm:dfm2}, we can always find at least one pair of masses $m_{i_0}, m_{i_1}$ ($\{i_0 \ne i_1\} \subset \N$), such that 
$$ |\qe_{i_0}(t)-\qe_{i_1}(t)|^{-1} - |q_{i_0}(t)-q_{i_1}(t)|^{-1} \le -C_2\ep^2, \;\; \forall t \in [t_1, t_2],$$
where $t_2-t_1 >0$ is independent of $\ep$. Then 
$$ \int_0^{1/2} U(\qe) - U(q) \,dt \le -C_3 \ep^2. $$
As a result, for $\ep>0$ small enough,
$$ \A_{1/2}(\qe) - \A_{1/2}(q) \le C_1 \ep^3 - C_3 \ep^2 <0.$$

When $t_0=n/4$, we define $\qe_0(t), t \in [0, n/4]$ as follows:
\begin{equation*}
\qe_0(t) = \begin{cases}
q_0(t) - \ep^2 \e_1, \;\; & \forall t \in [0, n/4-\ep], \\
q_0(t)+ (t-t_0)(2\ep-|t-t_0|), \;\; & \forall t \in[n/4, n/4-\ep].
\end{cases}
\end{equation*}
The rest is similar as above. 
\end{proof}

Since $\qo$ is strictly $x$ and $y$-monotone, as explained in Section \ref{sec:intro}, this immediately implies $\qo(t)$ is collision-free, for any $t \in (0, 1/2)$. By the proof of Lemma \ref{lm: x y dot >0}, $\qo(t)$ satisfies equation \eqref{eqn:nbody}, for any $t \in (0, 1/2).$ Meanwhile since only binary collisions are possible at $t \in \{0, 1/2\}$, it will be enough for us to prove $\qo(t)$, $t \in \{0, 1/2\}$, is free of binary collision as well. Before we proceed recall that $\qo(t)$ can have a binary collision if and only if
\begin{equation}
\label{eqn:binary collison z0}  z^{\om}_0(i/2)=0, \; \text{ for some } i \in \{0, \dots, [n/2]\}.
\end{equation}

\begin{lm} 
\label{lm:coll-free even}
If $n$ is even, $\qo(t),\; t \in \{0, 1/2\}$, is collision-free. 
\end{lm}
\begin{proof}
To simplify notation, let $q = \qo$. By a contradiction argument, let's assume
\begin{equation}
\label{eqn:z0} \exists\; i^* \in \{0, \dots, [n/2] \}, \text{ such that }  z_0(i^*/2) = 0,
\end{equation}
Using the deformation lemmas from Section \ref{sec:lemma}, we will find a contradiction. To ensure the $\om$-topological constraints will still be satisfied after the deformation, we need the precise value of $\om_{i^*}$. Without loss of generality, let's assume $\om_{i^*} =1$.  

Depending on the values of $n$ and $i^*$, six different cases need to be considered: in the first three cases, we assume $n = 4l$ and in the last three, $n=4l+2$, where $l \in \zz^+$. 

First let's assume $n =4l$. From Figure \ref{fig:even}(c), we can see, from $t=0$ to $t=1/2$, $m_j$, $j \in \I_a$, moves inside the quadrant $\mf{Q}_a$, where $a \in \{1, \dots, 4\}$ and $\I_a$'s are defined as follows:
\begin{align*}
\I_2 & = \{ 0, \dots, l-1 \} \cup \{ 7l, \dots, 8l-1 \}; \\
\I_1 & = \{ l, \dots, 2l-1 \} \cup \{ 6l, \dots, 7l -1 \}; \\
\I_4 & = \{ 2l, \dots, 3l-1 \} \cup \{ 5l, \dots, 6l-1 \}; \\
\I_3 & = \{ 3l, \dots, 4l-1 \} \cup \{ 4l, \dots, 5l -1 \}. 
\end{align*}
We point out that the motion of masses $\{m_i\}_{i \in \I_2}$, during $t \in [0,1/2]$, determines the motion of all the other masses during the same time through symmetries: for $m_i, i \in \I_3$, through $\mf{R}_{x}$; for $m_i, i \in \I_4$ through $\mf{R}_z$; for $m_i, i \in \I_1$ through $\mf{R}_{x} \circ \mf{R}_z$. Meanwhile $q$ being strictly $x$-monotone is equivalent to 
\begin{equation}
 \label{eqn:x-m} 
 \begin{cases} 
 & \forall 0 \le t_1 <t_2 \le 1/2, \;
  \begin{cases}
  x_i(t_1) < x_i(t_2), & \text{ if } i \in \{0, \dots, l-1\}; \\
  x_i(t_1) > x_i(t_2), & \text{ if } i \in \{7l, \dots, 8l-1 \},
  \end{cases}\\
  &\begin{cases}
   x_0(0) < x_0(\ey) = x_{8l-1}(\ey) &< x_{8l-1}(0)= x_1(0) < \cdots \\
  & \cdots < x_{l-1}(\ey)= x_{7l}(\ey) < x_{7l}(0)=0\\
  \end{cases} 
  \end{cases}
\end{equation}
and being strictly $y$-monotone is equivalent to 
\begin{equation}
 \label{eqn:y-m} 
 \begin{cases} 
 & \forall 0 \le t_1 <t_2 \le 1/2, \;
  \begin{cases}
  y_i(t_1) < y_i(t_2), & \text{ if } i \in \{0, \dots, l-1\}; \\
  y_i(t_1) > y_i(t_2), & \text{ if } i \in \{7l, \dots, 8l-1 \},
  \end{cases}\\
  &\begin{cases}
   0=y_0(0) < y_0(\ey) = y_{8l-1}(\ey) &< y_{8l-1}(0)= y_1(0) < \cdots \\
  & \cdots < y_{l-1}(\ey)= y_{7l}(\ey) < y_{7l}(0).\\
  \end{cases} 
  \end{cases}
\end{equation}
Again Figure \ref{fig:even}(c) may be helpful to see the above. 

\emph{Case 1}: $n =4l$ and  $i^* \in \{1, \dots, n/2-1 \}$. Then there is an isolated binary collision between $m_j$ and $m_k$ ($\{j > k \} \subset \I_2$) occurring inside the interior of $\mf{Q}_2$. Depend on the value of $i^*$, the collision may occur at $t=0$ or $t=1/2$:
$$ q_j(0) = q_k(0) =q_0(i^*/2), \;\text{ if } i^* \text{ is even},$$
$$ q_j(1/2) = q_k(1/2) =q_0(i^*/2), \; i^* \text{ is odd}. $$
Only the details for even $i$ will be given here, while the other is similar. By symmetric constraints, when $i^*$ is even, $j= 2n-i^*/2$ and $k=i^*/2$. To satisfy the $\om$-topological constraints with $\om_{i^*}=1$, after a local deformation near the isolated binary collision, we need $m_j$ below the $xy$-plane and $m_k$ above it at $t=0$. 


Following the notations used in Section \ref{sec:lemma}, for any $i \in \{j, k\}$, let $(r_i(t), \phi_i(t), \tht_i(t))$ be the spherical coordinates of $\mf{q}_i(t)$, where 
$$ \mf{q}_i(t) = q_i(t) - q_c(t), \; \text{ with } \; q_c(t) = [q_j(t) + q_k(t)]/2. $$
By Proposition \ref{prop:angle}, there is a $\phi^+_j \in [0, \pi]$ with $\lim_{t \to 0^+} \phi_j(t) = \phi^+_j.$ Different types of deformation lemmas will be needed, depends on the value of $\phi^+_j$.

If $\phi_j^+ \in (0, \pi]$, we proceed as the following: 

\emph{Step 1}. By Lemma \ref{lm:dfm1}, we make a local deformation of $(q_i(t))_{i \in \I_2}$, $t \in [0, 1/2]$, near the binary collision at $t=0$ to get a new path $(\qey_i(t))_{i \in \I_2}$, $t \in [0, 1/2]$, such that for $\ep_1>0$ small enough,
$$ \A_{\I_2, 1/2}(\qey) < \A_{\I_2, 1/2}(q). $$
Meanwhile $\qey_i(t),$ $t \in [0, 1/2]$, $i \notin \I_2,$ will be given through symmetries explained before. This gives us a new path $\qey(t)=(\qey_i(t))_{i \in \N}, t \in [0,1/2]$ satisfying
\begin{equation} 
\label{eqn: A Ia} \A_{\I_a, 1/2}(\qey) < \A_{\I_a, 1/2}(q), \;\; \forall a \in \{1, \dots, 4\}.
\end{equation}
At the same time, we have
\begin{equation}
 \label{eqn: Aa q} \A_{1/2}(q)= \sum_{a=1}^4 \A_{\I_a, 1/2}(q)+ \sum_{1 \le a< b\le4} \int_0^{1/2} U_{\I_a, \I_b}(q) \,dt;
\end{equation}
\begin{equation}
 \label{eqn: Aa qey} \A_{1/2}(\qey)= \sum_{a=1}^4 \A_{\I_a, 1/2}(\qey)+ \sum_{1 \le a< b\le4} \int_0^{1/2} U_{\I_a, \I_b}(\qey) \,dt,
\end{equation}
where 
$$U_{\I_a, \I_b}(q) = \sum_{i_1 \in \I_a, i_2 \in \I_b} |q_{i_1}-q_{i_2}|^{-1}; \;\; U_{\I_a, \I_b}(\qey) = \sum_{i_1 \in \I_a, i_2 \in \I_b} |\qey_{i_1}-\qey_{i_2}|^{-1}. $$

Besides the binary collision between $m_j$ and $m_k$ inside $\mf{Q}_2$, due to the symmetric constraints, there are three more binary collisions inside each one of the other three quadrants. Therefore by Lemma \ref{lm:dfm1}, when $m_{i}$ is one of those eight masses that are involved in the binary collisions, its path was deformed locally, when $t \in [0, t_0]$ for $t_0$ small enough (the paths of all the other masses are unchanged). As a result, this leads to changes in $U_{\I_a, \I_b}$, $1 \le a< b \le 4$. However due to the \emph{blow-up} technique used in the proof of Lemma \ref{lm:dfm1}, we can make the local deformations near the isolated binary collisions as small as we want, so the changes in $U_{\I_a, \I_b}$'s can be ignored compare to $\A_{\I_a, 1/2}$'s. Therefore by \eqref{eqn: A Ia}, \eqref{eqn: Aa q} and  \eqref{eqn: Aa qey}, we get $\A_{1/2}(\qey) < \A_{1/2}(q)$. 

\emph{Step 2}. Notice that after the local deformation by Lemma \ref{lm:dfm1}, it is not so clear $\qey$ is still $x$ and $y$-monotone. Meanwhile \eqref{eqn:x-m} and \eqref{eqn:y-m} clearly show $(q_i(t))_{i \in \I_2}$, $t \in [0,1/2]$, is $x$ and $y$-separated (by $m_j$ and $m_k$). Then we can make a deformation of $(\qey_i(t))_{i \in \I_2}$, $t \in [0,1/2]$, as in the proof of Lemma \ref{lm:mc}. This gives us a new path $(\qt_i^{\ep_1}(t))_{i \in \I_2}$, $t \in [0, 1/2]$. By what we prove in Lemma \ref{lm:mc}, it satisfies \eqref{eqn:x-m} and \eqref{eqn:y-m} (here we may need to shift the new path by some constants along the $x$ and $y$ direction to make sure $y_0(0)=x_{7l}(0)=0$). After this we define each $\qt^{\ep_1}_i(t)$, $ t \in [0, 1/2]$, $i \notin \I_2$, through the symmetries explained before. Then $(\qt^{\ep_1}_i(t))_{i \in \N},$ $t \in [0, 1/2]$ is $x$ and $y$-monotone, so it belongs to $\lmn_{\le, \om}$. 

By Lemma \ref{lm:mc},
$$ \A_{\I_i, \ey}(\qt^{\ep_1}) \le \A_{\I_i, \ey}(\qey) < \A_{\I_i, \ey}(q), \; \forall i \in \{1,\dots, 4 \}.$$ 
Furthermore the definition of $\qt^{\ep_1}$ given in the proof of Lemma \ref{lm:mc} implies
$$ |\qt^{\ep_1}_{i_1}(t) -\qt^{\ep_1}_{i_2}(t)| \ge |\qey_{i_1}(t) -\qey_{i_2}(t)|, \; \forall i_1 \in \I_a, \; \forall i_2 \in \I_b, \; \forall 1 \le a < b \le 4. $$
Therefore
$$ \sum_{1\le a< b\le 4} \int_0^{\ey} U_{\I_a, \I_b}(\qt^{\ep_1}) \,dt \le \sum_{1\le a< b\le 4} \int_0^{\ey} U_{\I_a, \I_b}(q^{\ep_1}) \,dt. $$
As a result $A_{1/2}(\qt^{\ep_1}) \le A_{1/2}(\qey) < A_{1/2}(q)$, which is absurd. 

If $\phi^+_j= 0$, Lemma \ref{lm:dfm1} can no longer be applied. However by \eqref{eqn:x-m} and \eqref{eqn:y-m}, the conditions of Lemma \ref{lm:dfm2} are satisfied now. Then by making a deformation of $(q_i(t))_{i \in \I_2}$, $t \in [0, 1/2]$ as in Lemma \ref{lm:dfm2}, we get a new path $(\qee_i(t))_{i \in \I_2}$, $t \in [0,1/2]$, which satisfies \eqref{eqn:x-m} and \eqref{eqn:y-m}, and $\A_{\I_2, 1/2}(\qee) < \A_{\I_2, 1/2}(q)$, for $\ep_2 >0$ small enough. Then by a similar argument as in \emph{Step 2}, we can find a $\qee \in \lmn_{\le, \om}$ with $\A_{1/2}(\qee) < \A_{1/2}(q)$, which is absurd. 

This finishes our proof of \emph{Case 1}.  

\emph{Case 2}: $n =4l$ and $i^* =0$. There is an isolated binary collision between $m_0$ and $m_{4l}$ at $t=0$ occurring inside the interior of $\mf{Q}_2 \cap \mf{Q}_3$. As $0 \in \I_2$ and $4l \in \I_3$, we need to consider the sub-system $(q_i(t))_{i \in \I_2 \cup \I_3}$, $t \in [0, 1/2]$. Because $(q_i(t))_{i \in \I_3}$ are determined by $(q_i(t))_{i \in \I_2}$ through $\mf{R}_x$, by \eqref{eqn:y-m}, $(q_i(t))_{i \in \I_2 \cup \I_3}, t \in [0,1/2]$, is $y$-separated by ($m_0$ and $m_{4l}$), although not $x$-separated anymore. However this is enough for us to apply the deformation lemmas from Section \ref{sec:lemma} and a contradiction can be reached by similar arguments used in \emph{Case 1}.  

\emph{Case 3}: $n=4l$ and $i^* = n/2$. There is an isolated binary collision between the masses $m_l$ and $m_{7l}$ at $t=0$ occurring inside the interior of $\mf{Q}_1 \cap \mf{Q}_2$. As $l \in \I_1$ and $7l \in \I_2$, we need to consider the sub-system $(q_i(t))_{i \in \I_1 \cup \I_2}$, $t \in [0, 1/2]$. Because $(q_i(t))_{i \in \I_1}$ is determined by $(q_i(t))_{i \in \I_2}$ through $\mf{R}_x \circ \mf{R}_z$, by \eqref{eqn:x-m}, $(q_i(t))_{i \in \I_1 \cup \I_2}$, $t \in [0, 1/2]$ is $x$-separated (by $m_l$ and $m_{7l}$), although not $y$-separated. However like in \emph{Case 2}, the deformation lemmas can still be applied to get a contradiction. 

This finishes our proof of the first three cases. For the remaining three cases, we assume $n=4l+2$. Now from $t=0$ to $t=1/2$, $m_j$, $j \in \J_a$, is moving inside $\mf{Q}_a$, for any $a \in \{1, \dots, 4\}$, where
\begin{align*}
\J_2 & = \{0, \dots, l \} \cup \{ 7l+4, \dots, 8l+3 \}; \\
\J_1 & = \{l+1, \dots, 2l \} \cup \{ 6l+3, \dots, 7l+3 \}; \\
\J_4 & = \{2l+1, \dots, 3l+1 \} \cup \{5l+3, \dots, 6l+2 \};\\
\J_3 & = \{3l+2, \dots, 4l+1 \} \cup \{ 4l+2, \dots, 5l+2 \}.
\end{align*}
The motion of masses $\{m_i\}_{i \in \J_2}$ during $t \in [0, 1/2]$ determines the motion of all the other masses during the same time interval through the symmetries as explained in the previous cases. Furthermore conditions similar to \eqref{eqn:x-m} and \eqref{eqn:y-m} can be given to help us show the concerned sub-system is $x$-separated, $y$ separated or both. Figure \ref{fig:even}(d) may be helpful to see all these. 

\emph{Case 4}: $n = 4l+2$ and $i^* \in \{1, \dots, n/2-1\}$. There is an isolated binary collision between $m_j$ and $m_k$ ($\{j>k\} \subset \J_2$) occurring inside $\mathring{\mf{Q}}_2$. The rest of is similar to \emph{Case 1}.

\emph{Case 5}: $n =4l+2$ and  $i^* = 0$. There is an isolated binary collision between $m_0$ and $m_{4l+2}$ ($0 \in \J_2$ and $4l+2 \in \J_3$) at $t=0$ occurring inside the interior of $\mf{Q}_2 \cap \mf{Q}_3$. Since $(q_i(t))_{i \in \J_2 \cup \J_3}, t \in [0, 1/2]$, is $x$-separated (by $m_0$ and $m_{4l+2}$), just repeat the same words from \emph{Case 2}.

\emph{Case 6}: $n =4l+2$ and $i^* = n/2$. There is an isolated binary collision between the mass $m_l$ and $m_{7l+3}$ ($l \in \J_2$ and $7l+3 \in \J_1$) at $t=1/2$ occurring inside the interior of $\mf{Q}_1 \cap \mf{Q}_2$.  As in \emph{Case 3}, $(q_i(t))_{i \in \J_1 \cup \J_2}$, $t \in [0, 1/2]$, is $x$-separated (by $m_l$ and $m_{7l+3}$), although the binary collision occurs at $t=1/2$, a contradiction can be reached by a similar argument. 

This finishes the last three cases and the entire proof. 
\end{proof}

\begin{lm}
\label{lm:odd} If $n$ is odd and $\om \in \Om_{[n/2]}$ satisfies condition \eqref{eqn:odd om}, $\qo(t),\; t \in \{0, 1/2\}$, is collision-free.  
\end{lm}

\begin{proof}
To simplify notation, let $\qo=q$. By a contradiction argument, like in the proof of Lemma \ref{lm:coll-free even}, let's assume \eqref{eqn:z0} holds. For odd $n$, because of the definition of the action of $h_2$, \eqref{eqn:z0} is equivalent to the following 
\begin{equation}
\label{eqn:coll odd} \exists\; i^* \in \{0, \dots, [n/2] \}, \text{ such that }  z_0(i^*) = 0.
\end{equation}
Notice that in \eqref{eqn:z0}, we have $z_0(i^*/2)=0$ instead of $z_0(i^*)=0$.

Depending on the values of $n$ and $i^*$, four different cases need to be considered. To distinguish them from those discussed in Lemma \ref{lm:coll-free even}, we will count from $7$. Figure \ref{fig:odd}(c) and Figure \ref{fig:odd}(d) will be helpful during the proof, as they indicate the motion of the masses for $t \in [0, 1/2]$, after projecting to the $xy$-plane.

First let's assume $n=4l+1$. From $t=0$ to $t=1/4$, $m_j$ is moving inside $\mf{Q}_1 \cup \mf{Q}_2$, if $j \in \I_5$; inside $\mf{Q}_3 \cup \mf{Q}_4$, if $j \in \I_6.$ Here
\begin{align*}
 \I_5 &= \{ 0, \dots, 2l \} \cup \{ 6l+2, \dots, 8l+1\}; \\
 \I_6 &= \{2l+1, \dots, 4l \} \cup \{ 4l+1, \dots, 6l+1 \}. 
\end{align*}
We observe that the motion of masses $\{m_i\}_{i \in \I_5}$ during $t \in [0,1/4]$ determines the motion of masses $\{m_i\}_{i \in \I_6}$ during the same time interval through $\mf{R}_x$. Furthermore $q$ being strictly $x$-monotone is equivalent to 
\begin{equation}
\label{eqn:x-m odd} \begin{cases}
& \forall 0 \le t_1 < t_2 \le \frac{1}{4}, \; \begin{cases}
x_i(t_1) < x_i(t_2), & \text{ if } i \in \{0, \dots, 2l \};\\
x_i(t_1) > x_i(t_2), & \text{ if } i \in \{6l+2, \dots, 8l+1\}, \\
\end{cases}\\
& \begin{cases}
x_0(0) < x_0(\frac{1}{4}) & < x_{8l+1}(\frac{1}{4})< x_{8l+1}(0) =x_1(0) < x_1(\frac{1}{4})< \cdots \\
& \cdots < x_l(\frac{1}{4})=0< \cdots< x_{6l+2}(0) = x_{2l}(0) < x_{2l}(\frac{1}{4}), \\
\end{cases}
\end{cases}
\end{equation}
and being strictly $y$-monotone is equivalent to 
\begin{equation}
\label{eqn:y-m odd} \begin{cases}
& \forall 0 \le t_1 < t_2 \le \frac{1}{4},\; \begin{cases}
y_i(t_1) < y_i(t_2), & \text{ if } i \in \{0, \dots, l\} \cup \{6l+2, \dots, 7l+1\};\\
y_i(t_1) > y_i(t_2), & \text{ if } i \in \{l+1, \dots, 2l \} \cup \{7l+2, \dots, 8l+1\}, \\
\end{cases}\\
& \begin{cases}
& y_0(0) < y_0(\frac{1}{4})= y_{2l}(\frac{1}{4})< y_{2l}(0) = y_{6l+2}(0) < y_{6l+2}(\frac{1}{4}) =  y_{8l+1}(\frac{1}{4}) < y_{8l+1}(0) \\ &= y_1(0)< \cdots \cdots < y_{7l+1}(\frac{1}{4}) = y_{7l+2}(\frac{1}{4}) < y_{7l+2}(0) = y_l(0) < y_l(\frac{1}{4}). \\
\end{cases}\\
\end{cases}
\end{equation}
It is easier to see the above connection by Figure \ref{fig:odd}(c).

\emph{Case 7}: $n = 4l+1$ and $i^* \in \{1, \dots, [n/2] \}$. Then by the symmetric constraints, there is an isolated binary collision between $m_{i^*}$ and $m_{2n-i^*}$ at $t=0$ occurring inside the interior of $\mf{Q}_1 \cup \mf{Q}_2$. By \eqref{eqn:x-m odd} and \eqref{eqn:y-m odd}, $(q_i(t))_{i \in \I_5}$, $t \in [0,1/4]$, is $x$ and $y$-separated by ($m_{2n-i^*}$ and $m_{i^*}$). Then a contradiction can be reached by a similar argument used in \emph{Case 1}. 
 
\emph{Case 8}: $n =4l+1$ and $i^* =0$. There is an isolated binary collision between $m_0$ and $m_{4l+1}$ occurring inside the interior of $\mf{Q}_2 \cap \mf{Q}_3$. Here we need to consider the entire system $\{m_i\}_{i \in \N}$, as $0 \in \I_5$ and $4l+1 \in \I_6$. Because $(q_i(t))_{i \in \I_6}$ are determined by $(q_i(t))_{i \in \I_5}$ through $\mf{R}_x$, by \eqref{eqn:x-m odd}, we can see $(q_i(t))_{i \in \I_5 \cup \I_6}$, $t \in [0, 1/4]$ is still $y$-separated (by $m_0$ and $m_{4l+1}$), although not $x$-separated. This is again enough for us to use the deformation lemmas to get a contradiction just like \emph{Case 2}. 

This finishes our proof for the cases with $n=4l+1$. From now on let's assume $n =4l+3$. Then from $t=0$ to $t= 1/4$, $m_i$ is moving inside $\mf{Q}_1 \cup \mf{Q}_2$, if $i \in \J_5$ and inside $\mf{Q}_3 \cup \mf{Q}_4$, if $i \in \J_6$, where
\begin{align*}
 \J_5 &= \{0, \dots, 2l+1\} \cup \{6l+5, \dots, 8l+5 \};\\
 \J_6 &= \{2l+2, \dots, 4l+2 \} \cup \{ 4l+3, \dots, 6l+4 \}. 
\end{align*}
Figure \ref{fig:odd}(d) may be helpful to see the above. Like the previous two cases the motion of masses $\{m_i\}_{i \in \J_5}$  during $t \in [0, 1/4]$ determines the motion of masses $\{m_i\}_{i \in \J_6}$ during the same time interval through the rotation $\mf{R}_x$ and conditions similar to \eqref{eqn:x-m odd} and \eqref{eqn:y-m odd} can be given to help us show the proper sub-system is $x$-separated, $y$-separated or both. 

\emph{Case 9}: $n=4l+3$ and $i^* \in \{1, \dots, [n/2]\}$. Just repeat the same argument used in \emph{Case 7}. 

\emph{Case 10}: $n=4l+3$ and $i^* \in 0$. Just repeat the same argument used in \emph{Case 8}. 

This finishes our entire proof. 
\end{proof}

Now we are ready to give a proof of the main theorem. 
\begin{proof}

[\textbf{Theorem \ref{thm:main}}] For any $\om \in \Omega_{[n/2]}$ satisfying \eqref{eqn:odd om}, by Proposition \ref{prop:coer}, there is a $\qo \in \Lmd^n_{\le, \om}$ with 
$$ \A_n(\qo)= \inf\{ \A_n(q)| \; q \in \Lmd^n_{\le, \om} \}. $$
Then Corollary \ref{cor: strictly x y monotone} shows $\qo$ is strictly $x$ and $y$-monotone, and Lemma \ref{lm:coll-free even} and \ref{lm:odd} show it is collision-free. Finally Lemma \ref{lm: x y dot >0} implies $\qo$ satisfies \eqref{eq: x dot >0} and \eqref{eq: y dot >0}, and it is a classical solution of \eqref{eqn:nbody}. 
\end{proof}

\section{Appendix} \label{sec:app}

The purpose of this section is to give a proof of Lemma \ref{lm:nondeg} using contradiction arguments. If Lemma \ref{lm:nondeg} does not hold, the action minimizer $\qo$ is a planar loop (all the masses always travel either in the $yz$-plane or the $xz$-plane). Then the symmetric and topological constraints implies $\qo$ must contain at least one isolated collision. By a local deformation near the isolated collision and a further deformation (to ensure the various constraints will be satisfied), we can find a new loop in $\lmn_{\le, \om}$ with action value strictly smaller than $\qo$'s and reach a contradiction.

Given a $q \in H^1([T_1, T_2], \rr^{3N})$, we say it is a \textbf{generalized solution} of \eqref{eqn:nbody}, if the set of collision moments $q^{-1}(\Delta)= \{t \in [T_1, T_2] |\; q(t) \in \Delta \}$ has measure zero, and $q(t)$ is a $C^2$ solution of \eqref{eqn:nbody} in $[T_1, T_2] \setminus q^{-1}(\Delta)$. For any $t_0 \in q^{-1}(\Delta)$, there is a $\I_0 \subset \N$, such that 
$$ q_i(t_0) =q_j(t_0), \; \forall \{i \ne j\} \subset \I_0; \;\; q_i(t_0) \ne q_k(t_0), \; \forall i \in \I_0, \; \forall k \in \N \setminus \I_0, $$
and we say $q(t_0)$ has an $\I_0$-cluster collision. Furthermore if there is a $\dl >0$ such that $([t_0 -\dl, t_0 +\dl] \cap [T_1, T_2]) \cap q^{-1}(\Delta) = \emptyset$, then we say $t_0$ is an isolated ($\I_0$-cluster) collision moment and $q(t_0)$ has an isolated ($\I_0$-cluster) collision. 

In the following, we introduce several local deformation lemmas that deform a planar collision solution near an isolated collision along the direction perpendicular to the plane containing it. The following set will be needed during the deformation
$$ \mf{T}=\big\{ \tau = (\tau_i)_{i \in \N}| \; \tau_i \in \{ 0, \pm 1 \} \text{ and } \tau \ne 0 \big\}.$$ 

\begin{lm}
\label{lm:dfm4} Given a generalized solution $q \in H^1([T_1, T_2], \rr^{3N})$ that is contained in the $yz$-plane, if $t_0 =T_1$ is an isolated $\I_0$-cluster collision moment and $\tau \in \mf{T}$ satisfies
\begin{equation}
\label{eqn:cond tau} \tau_{i_0} \ne \tau_{i_1}, \text{ for some } \{i_0 \ne i_1\} \subset \I_0, \text{ and } \tau_i = 0, \; \forall i  \in \N \setminus \I_0,
\end{equation} 
then for $\ep>0$ small enough, there exist an $h \in H^1([t_0, t_0 +\dl], \rr)$ and a path $\qe \in H^1([t_0, t_0+\dl], \rr^{3N})$, with $(\qe_i(t))_{i \in \N}=(q_i(t)+ \ep h(t) \tau_i \e_1)_{i \in \I},$ $t \in [t_0, t_0 +\dl]$, satisfying $\A_{t_0, t_0 +\dl}(\qe)< \A_{t_0, t_0 +\dl}(q)$ and the following:
\begin{enumerate}
\item[(a).] $h(t)=1$, $\forall t \in [t_0, t_0 +\dl_1]$, for some $0< \dl_1=\dl_1(\ep)< \dl$ small enough; 
\item[(b).] $h(t) =0$, $\forall t \in [t_0 +\dl_2, t_0 +\dl]$, for some $\dl_1< \dl_2=\dl_2(\ep) < \dl$ small enough;
\item[(c).] $h(t)$ is decreasing for $t \in [t_0 + \dl_1, t_0 +\dl_2]$. 
\end{enumerate}
\end{lm}

\begin{lm}
\label{lm:dfm5} Given a generalized solution $q \in H^1([T_1, T_2], \rr^{3N})$ that is contained in the $yz$-plane, if $t_0 =T_2$ is an isolated $\I_0$-cluster collision moment and $\tau \in \mf{T}$ satisfies condition \eqref{eqn:cond tau}, then for $\ep>0$ small enough, there exist $h \in H^1([t_0-\dl, t_0], \rr)$ and a path $\qe \in H^1([0, T], \rr^{3N})$, with $(\qe_i(t))_{i \in \N}=(q_i(t)+ \ep h(t) \tau_i \e_1)_{i \in \N}$, $t \in [t_0-\dl,t_0]$, satisfying $\A_{t_0-\dl, t_0}(\qe)< \A_{t_0-\dl, t_0}(q)$ and the following:
\begin{enumerate}
\item[(a).] $h(t)=1$, $\forall t \in [t_0-\dl_1, t_0]$, for some $0< \dl_1=\dl_1(\ep)< \dl$ small enough; 
\item[(b).] $h(t) =0$, $\forall t \in [t_0-\dl, t_0-\dl_2]$, for some $\dl_1< \dl_2=\dl_2(\ep) < \dl$ small enough;
\item[(c).] $h(t)$ is increasing for $t \in [t_0-\dl_2, t_0-\dl_1]$. 
\end{enumerate}
\end{lm}
While the above two lemmas can be used on isolated collisions happening at the boundaries of a fundamental domain, the following is for collisions occurring in the interior. 

\begin{lm}
\label{lm:dfm6} Given a generalized solution $q \in H^1([T_1, T_2], \rr^{3N})$ that is contained in the $yz$-plane, if $t_0 \in (T_1, T_2)$ is an isolated $\I_0$-cluster collision moment and $\tau \in \mf{T}$ satisfies condition \eqref{eqn:cond tau}, then for $\ep>0$ small enough, there exist an $h \in H^1([t_0-\dl,t_0 + \dl], \rr)$ and a path $\qe \in H^1([t_0 -\dl, t_0 + \dl], \rr^{3N})$, with $(\qe_i(t))_{i \in \N}=(q_i(t)+ \ep h(t) \tau_i \e_1)_{i \in \N}, t \in [t_0-\dl, t_0 +\dl],$ satisfying $\A_{t_0-\dl, t_0 + \dl}(\qe)< \A_{t_0-\dl, t_0 +\dl}(q)$ and the following:
\begin{enumerate}
\item[(a).] $h(t)=1$, $\forall t \in [t_0-\dl_1, t_0+\dl_1]$, for some $0< \dl_1=\dl_1(\ep)< \dl$ small enough; 
\item[(b).] $h(t) =0$, $\forall t \in [t_0 -\dl, t_0-\dl_2] \cup [t_0 +\dl_2, t_0 +\dl]$, for some $\dl_1< \dl_2=\dl_2(\ep) < \dl$ small enough;
\item[(c).] $h(t)$ is increasing for $t \in [t_0-\dl_2, t_0 -\dl_1]$;
\item[(d).] $h(t)$ is decreasing for $t \in [t_0 +\dl_1, t_0 +\dl_2]$.  
\end{enumerate}
\end{lm}
\begin{rem}
\begin{enumerate}
\item[(i).] When we apply the above lemmas, duo to the symmetric constraints imposed by certain group $G$, there may be isolated cluster collisions different from $\I_0$ at the same moment. In particular the corresponding $\tau$ needs to be chosen, so that the path still satisfies the symmetric constraints after the local deformation, as a result condition \eqref{eqn:cond tau} should be replaced by the following:
\begin{itemize}
\item $\tau_{i_0} \ne \tau_{i_1}$, for some $\{ i_0 \ne i_1\} \subset \I_0$;
\item $\tau_{i_1}\e_1 =\rho(g)(\tau_{i_0}\e_1)$, if $i_1 =\sg(g^{-1})(i_0) $, for some $i_0 \in \I_0$ and $g \in \ker(t_0)=\{g \in G|\; \tau(g^{-1})(t_0)= t_0 \}$; 
\item $\tau_i=0$, $\forall i \in \N \setminus \bigcup_{g \in \ker(t_0)}(\sg(g^{-1})(\I_0)).$
\end{itemize}
\item[(ii).] Although we stated the above lemmas for $q$ contained in the $yz$-plane, similar results hold as well if $q$ is contained in the $xz$-plane or $xy$-plane. One just need to replace $\e_1$ by $\e_2$ or $\e_3$ correspondingly in those lemmas.  

\end{enumerate}
\end{rem}
Similar local deformation lemmas for $q$ contained in a one dimensional subspace were proved in Lemma 4.1, 4.2 and 4.3 in \cite{Y15c}. The same proof can be generalized to our cases without any difficulty, see \cite[Remark 4.1]{Y15c}. We will not repeat the details here. The key point is that we only make deformations along a direction that is perpendicular to the subspace where $q$ belongs. 

\begin{lm}
\label{lm: isolated collision even} When $n$ is even ($n=2\ell$), for any $\om \in \Om_{[n/2]}$, let $q \in \Lmd^n_{\le, \om}$ be a generalized solution of \eqref{eqn:nbody}, which is contained either in the $yz$-plane or the $xz$-plane. If $t_0 \in (0, 1/2)$ is an isolated collision moment of $q$, then there is a $\qt \in \Lmd^n_{\le, \om}$ satisfying $\A_{1/2}(\qt) < \A_{1/2}(q). $ 
\end{lm}
\begin{proof}
We will only give the details for the case that $q$ is contained in the $yz$-plane, while the other can be proven similarly. Due to the symmetric constraints, there may be more than one collision clusters at $t=t_0$. However there always exists an $\I_0$-cluster collision at $t=t_0$ satisfying $\I_0 \cap \I \ne \emptyset$, where 
\begin{equation}
\label{eq: I isolated collision} \I = \{0, \dots, [n/4]\} \cup \{2n -[n/4], \dots, 2n-1\}. 
\end{equation}

First assume there is a $j \in \I_0 \cap \{0, \dots, [n/4]\} \ne \emptyset$. We choose a $\tau \in \mf{T}$ with each $\tau_i = 0$, except the following four. 
$$ \tau_{j} = \tau_{2\ell +j} = -1, \; \tau_{\ell + j} = \tau_{3\ell+j} =1. $$ 
Then by lemma \ref{lm:dfm6}, we can make a local deformation of $q$ along the direction perpendicular to the $yz$-plane near $t_0$ to get a new path $\qe(t)$, $t \in [0, 1/2]$ with $\A_{1/2}(\qe) < \A_{1/2}(q)$. Since $t_0 \in (0, 1/2)$ and only the path of $m_{i\ell +j}$, $i=0, 1, 2,3$, are deformed, $\qe$ still satisfies the symmetric, $\om$-topological constraints and the $y$-monotone constraints, but violates the $x$-monotone constraints. Because of this we make a further deformation of $\qe$ as follows:
\begin{equation} \label{eq: qt j1}
 \qt_j(t) = \begin{cases}
 \qe_j(t) + 2 (x^{\ep}_j(t_0)- x^{\ep}_j(t))\e_1, \;\; & \forall t \in [0, t_0], \\
 \qe_j(t), \;\; & \forall t \in [t_0, 1/2], 
 \end{cases}
 \end{equation} 
 and for any $t \in [0, 1/2]$
 \begin{equation} \label{eq: qt i1}
 \qt_i(t) = \begin{cases}
 \qe_i(t) + \tilde{x}_{j}(0), \; & \text{ if } i \in \{0, \dots, j-1\} \cup \{2n-j, \dots, 2n-1\}, \\
 \qe_i(t), \; & \text{ if } i \in \{j+1, \dots, [\frac{n}{4}]\} \cup \{2n - [\frac{n}{4}], \dots, 2n-1-j\}, 
 \end{cases}. 
 \end{equation}
As usual, the paths of $m_i$, $i \in \N \setminus \I$ follows from symmetric constraints once we have above. Now the new path $\qt$ satisfies the $x$-monotone constraints as well and belongs to $\Lmd^n_{\le, \om}$. Furthermore the above deformation preserves the kinetic energy and do not increase the potential energy, so $\A_{1/2}(\qt) \le \A_{1/2}(\qe) < \A_{1/2}(q)$. 

Now let's assume there is a $j \in \I_0 \cap \{2n - [n/4], \dots, 2n-1 \} \ne \emptyset$, we apply Lemma \ref{lm:dfm6} with a $\tau \in \mf{T}$ satisfying $\tau_i =0$, for every $i$ except the following four, 
$$ \tau_j = \tau_{j - 2\ell} = -1, \; \; \tau_{j-\ell} = \tau_{j -3\ell} =1. $$
The rest of argument is exactly the same as above, except $(\qt_i(t))_{i \in \I}$, $t \in [0, 1/2]$, should be defined as follows: 
\begin{equation} \label{eq: qt j2}
\qt_j(t)= \begin{cases}
\qe_j(t), \; \; & \forall t \in [0, t_0], \\
\qe_j(t) + 2(x^{\ep}_j(t_0)-x^{\ep}_j(t))\e_1, \;\; & t \in [t_0, 1/2], 
\end{cases}
\end{equation}
and for any $t \in [0, 1/2]$
\begin{equation} \label{eq: qt i2}
\qt_i(t) = \begin{cases}
\qe_i(t) + \tilde{x}_j(1/2), \; & \text{ if } j \in \{0, \dots, 2n-1-j\} \cup \{ j+1, \dots, 2n-1 \}, \\
\qe_i(t), \; & \text{ if } i \in \{2n-j, \dots, [\frac{n}{4}]\} \cup \{2n - [\frac{n}{4}], \dots, j-1 \}. 
\end{cases}
\end{equation}
\end{proof}

\begin{lm}
\label{lm: isolated collision odd} When $n$ is odd ($n=2\ell+1$), for any $\om \in \Om_{[n/2]}$, let $q \in \Lmd^n_{\le, \om}$ be a generalized solution of \eqref{eqn:nbody}, which is contained either in the $yz$-plane or the $xz$-plane. If $t_0 \in (0, 1/4)$ is an isolated collision moment of $q$, then there is a $\qt \in \Lmd^n_{\le, \om}$ satisfying $\A_{1/4}(\qt) < \A_{1/4}(q).$ 
\end{lm}

\begin{proof}
We will only give the details for the case that $q$ is contained in the $yz$-plane, while the other can be proven similarly. Notice that due to the symmetric constraint, the existence of an isolated collision moment $t_0 \in (0, 1/4)$ is equivalent to the existence of an isolated collision moment $t_1 \in (0, 1/2) \setminus \{1/4\}$.  

The rest of the proof is the same as Lemma \ref{lm: isolated collision even}, except that when we apply Lemma \ref{lm:dfm6}, if $j \in \I_0 \cap \{0, \dots, [n/4]\} \ne \emptyset$, then the proper $\tau \in \mf{T}$ should be
$$ \tau_j = \tau_{j +n} =-1, \text{ and } \tau_i = 0, \; \forall i \in \N \setminus \{j, j+n\}, $$
and if $j \in \I_0 \cap \{2n -[n/4], \dots, 2n-1 \}\ne \emptyset$, then the proper $\tau \in \mf{T}$ should be 
$$ \tau_j = \tau_{j -n} = -1, \text{ and } \tau_i =0, \; \forall i \in \N \setminus \{j, j-n \}. $$
\end{proof}

Now we will give a proof of Lemma \ref{lm:nondeg} through the following three lemmas. 

\begin{lm}
\label{lm:B1} For any $\om \in \Om_{[n/2]}$, if $\qo$ is a minimizer of the action functional in $\lmn_{\le, \om}$, then $x_0^{\om}(0)<0$. 
\end{lm}
\begin{proof}
For simplicity, let $q =\qo$. By a contradiction argument, assume the desired result does not hold. Then $x_i(t) =0$, for all $i$ and $t$, as $q$ is $x$-monotone. Hence the masses stay on the $yz$-plane all the time.

Notice that as $q(t)$ is a minimizer of the action functional in $\Lmd^n_{\le, \om}$, it is a generalized solution of \eqref{eqn:nbody} (see \cite[Section 5]{FT04}. A few explanation may be needed as the $x$-monotone and $y$-monotone constraints were required here. First under the above assumption $q$ belongs to the $yz$-plane, so the $x$-monotone constraints will not be a problem. For the $y$-monotone constraint, if $y_0(n/4)=0$, then $q$ being $y$-monotone implies the $y$ component of every $q_i$ is zero. Just like what said about the $x$-monotone constraint, it will not be a problem. If $y_0(n/4)>0$, by Proposition \ref{prop:smc}, $q$ must be strictly $y$-monotone. Then by a similar argument used in the proof of Lemma \ref{lm: x y dot >0}, we can show it satisfies \eqref{eq: y dot >0}. Again the $y$-monotone constraint will not be a problem.   

In the following we will show our assumption implies the existence of at least one isolated collision. Then we can find a new path in $\lmn_{\le, \om}$ with action value strictly smaller than $q$'s, which is a contradiction. 


Recall that for even $n$ $(n=2\ell)$, the action of $h_2$ implies, 
$$ (x_0, y_0, z_0)(t) = (-x_{\ell}, -y_{\ell}, z_{\ell})(t),\;\; \forall t.$$ 
By our assumption, $x_0(0)=0$ and by \eqref{eqn:q0symmetry}, $y_0(0)=0$. Therefore $q_0(0)=q_{\ell}(0)$. As a result, $0$ is an $\I_0$-cluster collision moment, with at least $\{0, \ell\} \subset \I_0$. We may assume it is also isolated, as otherwise by results from \cite[Section 5]{FT04}, there must be an isolated collision moment in $(0, 1/2)$ close to $0$, then a contradiction can be reached by Lemma \ref{lm: isolated collision even}. Now choose a $\tau\in \mf{T}$ with
\begin{equation}
 \label{eqn:tau} \tau_0=\tau_n=-1, \tau_{\ell} =\tau_{\ell+n} =1, \text{ and } \tau_i =0, \; \forall i \in \N \setminus \{0, \ell, n, \ell+ n \}.
 \end{equation} 
 We can use Lemma \ref{lm:dfm4} to make a local deformation of $q(t)$, $t \in [0, 1/2]$, near the collision moment $t=0$. Notice that only the paths of $m_i$, $i\in \{0, \ell, n, n+\ell \}$ are deformed during the above process. Figure \ref{fig:dfm1}(a) shows how they are deformed. 

\begin{figure}
  \centering
  \includegraphics[scale=0.70]{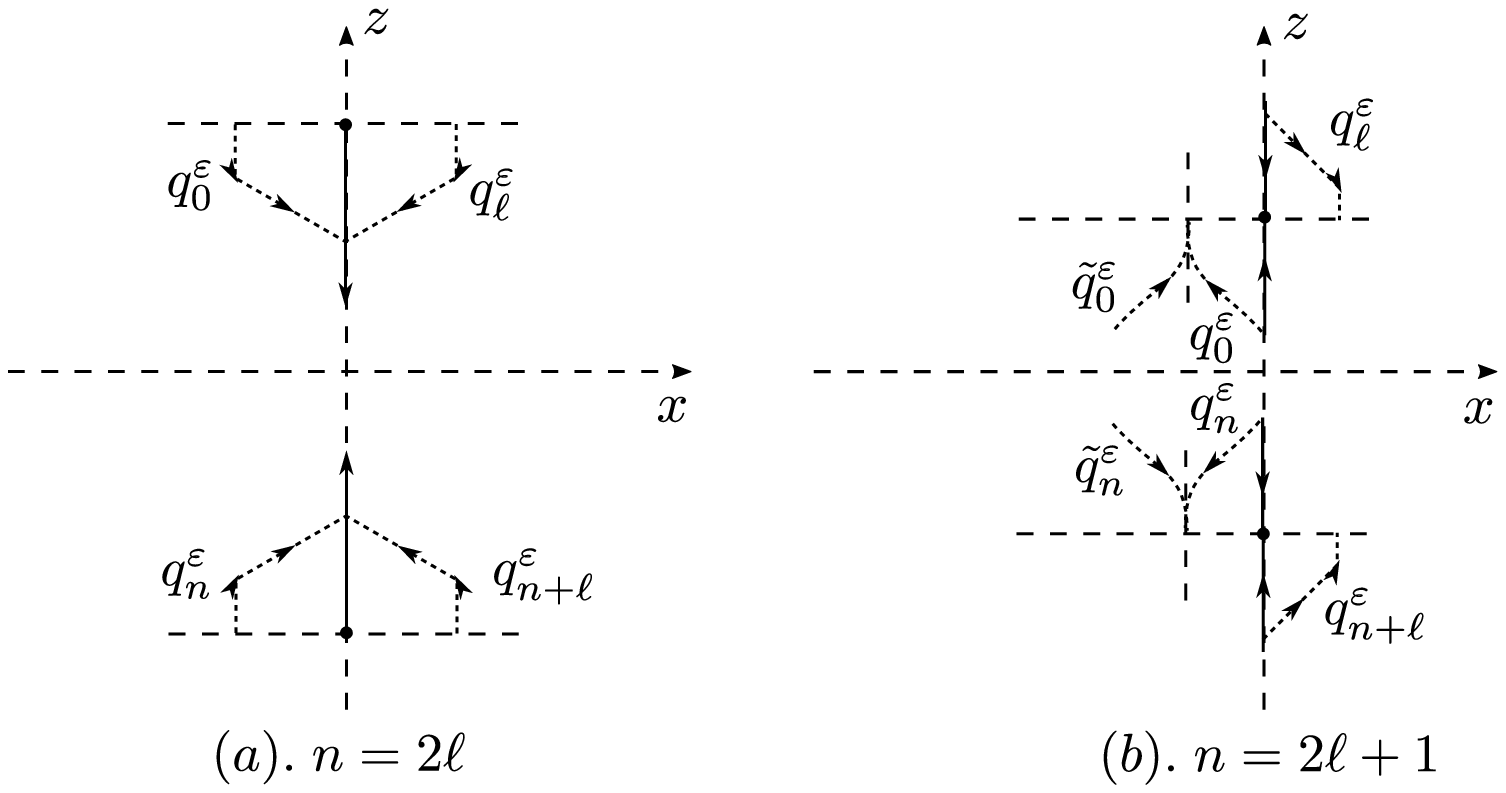}
  \caption{}
  \label{fig:dfm1}
\end{figure}

As we can see the deformed path is $x$ and $y$-monotone, and the $\om$-topological constraints are preserved as well. Therefore $\qe \in \lmn_{\le, \om}$. By Lemma \ref{lm:dfm4}, $\A_{1/2}(\qe)< \A_{1/2}(q)$, which is a contradiction. 

When $n$ is odd $(n=2\ell+1)$, the action of $h_2$ implies, 
$$(x_0, y_0, z_0)(1/2-t)= (-x_\ell, y_\ell, z_\ell)(t), \;\; \forall t.$$
By our assumption, $x_0(1/4)=x_{\ell}(1/4)=0$, so $q_0(1/4)=q_{\ell}(1/4)$ and $1/4$ is an $\I_0$-cluster collision moment with at least $\{0, \ell\} \subset \I_0$. Again we assume it is also isolated, otherwise there is a isolated collision moment nearby and by Lemma \ref{lm: isolated collision odd}, there is a contradiction. Let $\tau$ be the same as above. By Lemma \ref{lm:dfm5}, we make a local deformation of $q(t)$, $t \in [0, 1/4]$, near $t=1/4$, and it gives us a new path $\qe(t)$, $t \in [0, 1/4]$ with $\A_{1/4}(\qe) < \A_{1/4}(q)$. 

Only the paths of $m_i$, $i \in \{0, \ell, n, n+\ell\}$ were deformed during the above process. The $\om$-topological constraints and the property of being $y$-monotone were preserved, see Figure \ref{fig:dfm1}(b). However for odd $n$, being $x$-monotone means the above four masses can not move backward in the $x$-direction from $t=0$ to $t=1/4$. As we can see from Figure \ref{fig:dfm1}(b), $\qe_0(t)$ and $\qe_n(t)$ violate this. Therefore we make a further deformation and define a new path $\qt^{\ep}(t), t \in [0, 1/4]$ as 
$$ 
\qt^{\ep}_i(t) = 
\begin{cases}
\qe_i(t), & \text{ if } i \in \N \setminus \{0, n\}, \\
\R_{yz}\qe_i(t) + 2 x^{\ep}_i(1/4)\e_1, & \text{ if } i \in \{0, n \},  \\
\end{cases} 
\;\;\;\forall t \in [0,1/4].
$$
Again see Figure \ref{fig:dfm1}(b) for an illuminating picture. Now $\qt^{\ep}$ is $x$ and $y$-monotone and is contained in $\lmn_{\le, \om}$. By the definition of $\qt^{\ep}$, $\A_{1/4}(\qt^{\ep}) \le \A_{1/4}(q^{\ep})$. Then $\A_{1/4}(\qt^{\ep})< \A_{1/4}(q)$, which is absurd. 
\end{proof}

\begin{rem} \label{rem:strictly x mono}
For any $\om \in \Om_{[n/2]}$ (it does not need to satisfy condition \eqref{eqn:odd om}), Lemma \ref{lm:B1} and Proposition \ref{prop:smc} immediately implies any action minimizer $\qo \in \lmn_{\le, \om}$ must be strictly $x$-monotone, so it cannot be a planar path in the $yz$-plane.
\end{rem}

\begin{lm}
\label{lm:B2} When $n$ is even $(n=2\ell)$, for any $\om \in \Om_{[n/2]}$, if $\qo$ is a minimizer of the action functional in $\lmn_{\le, \om}$, then $y^{\om}_0(n/4)>0$.
\end{lm}
\begin{proof}
For simplicity, let $q =\qo$. By a contradiction argument, assume the statement is not true. Then $q$ being $y$-monotone implies the masses must travel on the $xz$-plane all the time, so $y_i(t)=0$, for all $i$ and $t$. Meanwhile by Remark \ref{rem:strictly x mono}, $q$ must be strictly $x$-monotone. First by the same argument given at the beginning of the previous proof, $q$ is a generalized solution of \eqref{eqn:nbody}. 

Depending on whether $\ell$ is even or odd, the proof will be slightly different. 

First let's assume $\ell =2l$, by the action of $h_2$, for any $t$
\begin{align*}
(x_l, y_l, z_l)(t)&=(-x_{3l}, -y_{3l}, z_{3l})(t);\\
(x_{5l}, y_{5l}, z_{5l})(t) &= (-x_{7l}, -y_{7l}, z_{7l})(t).
\end{align*}
Let $\I_0= \{l, 3l, 5l, 7l\}$. For any $i \in \I_0$, $y_i(0)=0$ by our assumption. By \eqref{eqn:q0symmetry} and the symmetric constraints, 
$$ x_l(0)=x_{5l}(0)=x_0(n/4)=0; \;\; x_{3l}(0) = x_{7l}(0)= x_0(3n/4)=0. $$
Hence $q$ has an $\I_0$-cluster collision at $t=0$. By the same argument given in the previous proof, we may assume it is also isolated. Since $q$ is strictly $x$-monotone, no other mass can be involved in this isolated collision cluster. To see these, one can project the paths in Figure \ref{fig:even}(c) to the $x$-axis.

Choose a $\tau=(\tau_i)_{i \in \N}$ with
$$ \tau_l=\tau_{7l}=1, \; \tau_{3l}=\tau_{5l}=-1, \; \text{ and } \tau_i =0, \; \forall i \in \N \setminus \I_0.$$
By Lemma \ref{lm:dfm4}, there is a new path $\qe(t), t \in [0, 1/2]$ with $\A_{1/2}(\qe) < \A_{1/2}(q)$ (see Figure \ref{fig:dfm2}(a)). As we can see only the paths of $m_i$, $i \in \I_0$, were changed in the above process, and $\qe$ satisfies the required symmetric, monotone and topological constraints, so it is contained in $\lmn_{\le, \om}$, which is a contradiction. 

\begin{figure}
  \centering
  \includegraphics[scale=0.70]{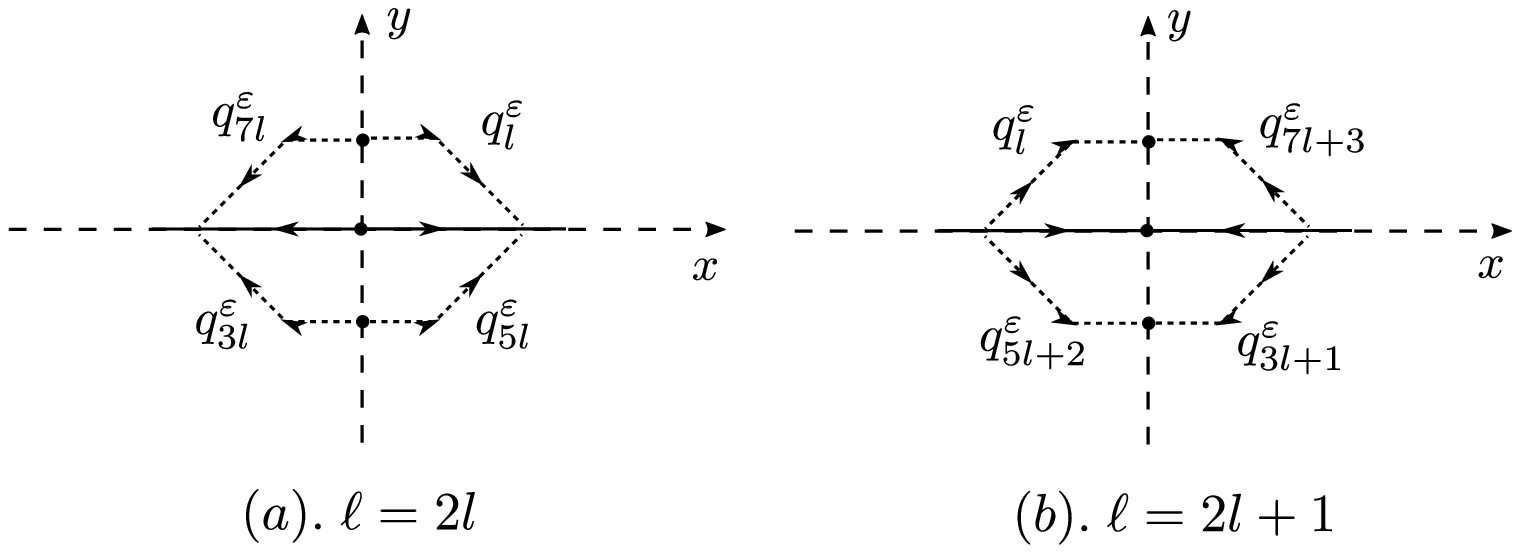}
  \caption{}
  \label{fig:dfm2}
\end{figure}

When $\ell= 2l+1,$ by the action of $h_2$,
\begin{align*}
(x_l, y_l, z_l)(t) &=(-x_{3l+1}, -y_{3l+1}, z_{3l+1})(t); \\
(x_{5l+2}, y_{5l+2}, z_{5l+2})(t) &= (-x_{7l+3}, -y_{7l+3}, z_{7l+3})(t).
\end{align*}
Now let's assume $\I_0=\{l, 3l+1, 5l+2, 7l+3\}.$ Then for any $i \in \I_0$, $y_i(0)=0$ by our assumption. By \eqref{eqn:q0symmetry} and the symmetric constraints
$$ x_l(1/2)=x_{5l+2}(1/2) = x_0(n/4)=0; \;\; x_{3l+1}(1/2)= x_{7l+3}(1/2)=x_0(3n/4)=0. $$
Therefore $q$ has an $\I_0$-cluster collision at $t =1/2$. Again we may assume it is isolated. Like above no other masses are involved in this isolated collision cluster, as $q$ is strictly $x$-monotone. Again one could project the path in Figure \ref{fig:even}(d) to the $x$-axis to see these.

Now we choose a $\tau=(\tau_i)_{i \in \N}$ with
$$ \tau_l=\tau_{7l+3}=1, \; \tau_{3l+1}=\tau_{5l+2}=-1, \; \text{ and } \tau_i =0, \; \forall i \in \N \setminus \I_0.$$
By similar arguments used as above (instead of Lemma \ref{lm:dfm4}, Lemma \ref{lm:dfm5} should be used now), we can find a new path $\qe \in \lmn_{\le, \om}$ (see Figure \ref{fig:dfm2}(b)) with $\A_{1/2}(\qe) < \A_{1/2}(q)$, which is a contradiction.
\end{proof}

The above lemma only considers even $n$, we still need to prove the same result for odd $n$. We point out that for odd $n$, even when all the masses travel in the $xz$-plane the entire time, the symmetric constraints alone do not guarantee the existence of a collision. However when we take the $\om$-topological constraints into consideration as well, then this can be guaranteed. This is also why condition \eqref{eqn:odd om} was required.

\begin{lm}
\label{lm:B3} When $n$ is odd $(n=2\ell +1)$, for any $\om \in \Om_{[n/2]}$ satisfies condition \eqref{eqn:odd om}, if $\qo$ is a minimizer of the action functional in $\lmn_{\le, \om}$, then $y_0^{\om}(n/4)>0$. 
\end{lm}
\begin{proof}
For simplicity let $q=\qo$. By a contradiction argument, let's assume $y_0(n/4)=0$, then the masses must always stay in the $xz$-plane. By the same argument given at the beginning of the proof of Lemma \ref{lm:B2}, $q(t)$ is strictly $x$-monotone, and it is a generalized solution of \eqref{eqn:nbody}. 

By condition \eqref{eqn:odd om}, there is a $1 \le k \le [n/2]-1$, such that $\om_k \ne \om_{k+1}$. Without loss of generality, we assume $\om_k=-1$ and $\om_{k+1}=1$ for the rest of the proof. 

Depending on whether $k$ is even or odd, the proof will be slightly different. First let's assume $k$ is even. From \eqref{eqn:g1} and \eqref{eqn:tc}, we get 
\begin{align*}
z_{k/2}(0) &= z_0(k/2) = \om_k |z_0(k/2)| = - |z_0(k/2)| \le 0; \\
z_{k/2}(1/2) &= z_0(\frac{k+1}{2}) = \om_{k+1}|z_0(\frac{k+1}{2})|= |z_0(\frac{k+1}{2})| \ge 0, 
\end{align*}
This implies $z_{k/2}(t_0)=0$, for some $t_0 \in [0, 1/2]$. Recall the action of $h_1$ implies, 
$$ (x_{k/2}, y_{k/2}, z_{k/2})(t) = (x_{k/2+n}, -y_{k/2+n}, -z_{k/2+n})(t), \;\; \forall t. $$
Hence $z_{k/2+n}(t_0)=-z_{k/2}(t_0)=0.$ Meanwhile our assumption implies $y_{k/2}(t_0)= y_{k/2 +n}(t_0)=0$. As a result, $q_{k/2}(t_0) = q_{k/2 +n}(t_0)$ belong to the $x$-axis, and $q$ has a collision involving at least $m_{k/2}$ and $m_{k/2+n}$ at the moment $t_0$. We may assume such a collision is isolated, otherwise there must be an isolated collision moment close to $t_0$ and a contradiction can be reach by Lemma \ref{lm: isolated collision odd}.  

As $t_0 \in [0, 1/2]$, although $[0, 1/4]$ is a fundamental domain of $q \in \lmn_{\le, \om}$. We will focus on the motion of the sub-system $\{m_i\}_{i \in \I}$ from $t=0$ to $t=1/2$, where 
\begin{align*}
 \I= & \{0, \dots, [\ell/2]\} \cup \{n-[(\ell+1)/2], \dots, n-1\} \cup \\
   &  \{ n, \dots, n+[\ell/2]\} \cup \{ 2n-[(\ell+1)/2], \dots, 2n-1\}.
 \end{align*} 
By this we mean the local deformation lemmas will be applied to $(q_i(t))_{i \in \I}$, $t \in [0, 1/2]$, while for any $i \in \N \setminus \I$, $q_i(t)$, $t \in [0, 1/2]$ will be given by the action of $h_2$. 
 
Depending on the value of $t_0$, three different cases need to be considered and correspondingly three different disjoint partitions of the index set $\I= \cup_{j=0}^{3}\I_{ij}$, $i =1,2$ and $3,$ will be given. Notice that $\{k/2, k/2 +n\} \subset \I. $

\emph{Case 1:} $t_0 \in (0, 1/2)$. Let $\I_{10}=\{k/2, k/2+n\}$. Since $q$ is strictly $x$-monotone, for any $i \in \I \setminus \I_{10}$, one of the following must hold
$$ x_i(t_0) < x_{k/2}(0)< x_{k/2}(t_0); \;\; \;\; x_i(t_0)> x_{k/2}(1/2) > x_{k/2}(t_0).$$ 
Therefore no other mass can collide with $m_{k/2}$ and $m_{k/2+n}$ at the moment $t_0$ and $q$ has an isolated $\I_{10}$-cluster collision at $t_0$. 

Care has to be taken during the deformation, so that the monotone constraints will not be violated after that. For this purpose, choose an arbitrary $\qt \in \lmn_{<, \om}$ (for the rest of the proof), we define the remaining $\I_{1j}$'s as follows: 
\begin{align*}
\I_{11} &= \{i \in \I|\; \yt_i(t) > \yt_{k/2}(1/2), \;\; \forall t \in (0, 1/2) \};\\
\I_{12} &= \{i \in \I|\; \yt_{k/2+n}(0) < \yt_i(t) < \yt_{k/2}(0), \;\; \forall t \in (0, 1/2)\};\\
\I_{13} &= \{i \in \I|\; \yt_i(t) < \yt_{k/2+n}(1/2), \;\; \forall t \in (0,1/2)\}.
\end{align*}
The above definitions obviously is independent of the choice of $\qt \in \lmn_{<, \om}$. 

Now choose a $\tau= (\tau_i)_{i \in \I}$ satisfying 
$$ \tau_{k/2}=1, \; \tau_{k/2 +n}=-1, \; \text{ and }\; \tau_i =0, \; \forall i \in \I \setminus \I_{10}. $$
With such a $\tau$, we can get a new path $(\qe_i(t))_{i \in \I}, t \in [0,1/2]$ by making a local deformation of $(q_i(t))_{i \in \I}$ near $t_0$ based on Lemma \ref{lm:dfm6} (see Figure \ref{fig:dfm3}(a)). 
\begin{figure}
  \centering
  \includegraphics[scale=0.80]{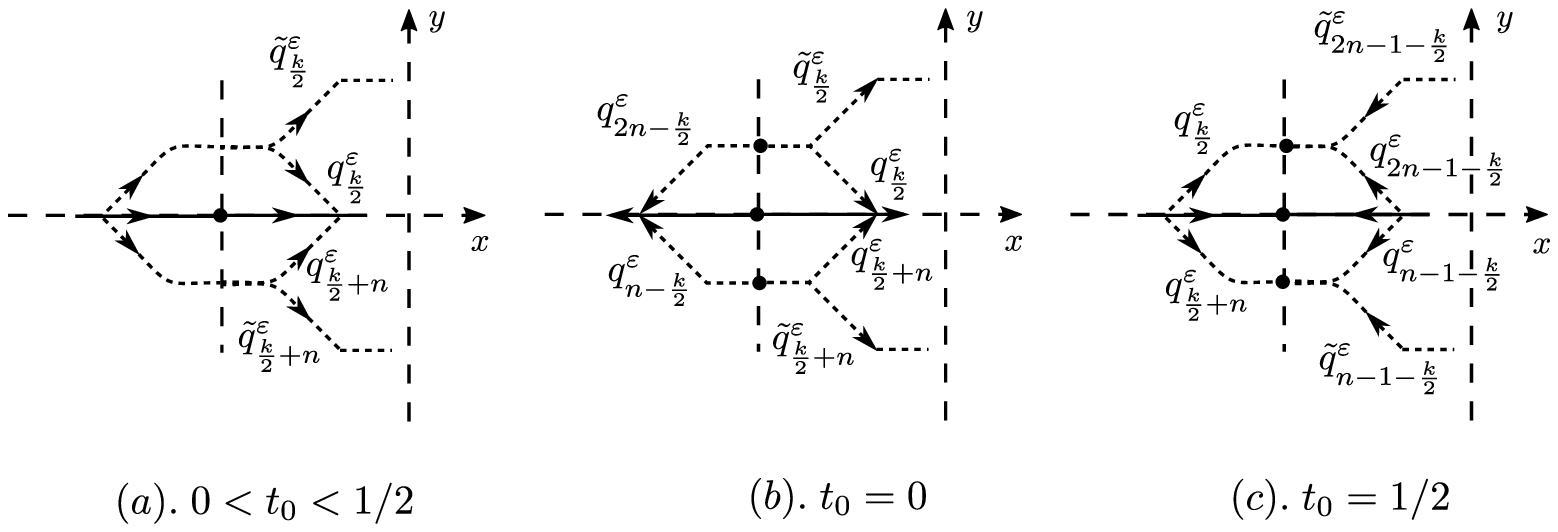}
  \caption{}
  \label{fig:dfm3}
\end{figure}

The new path clearly can not be $y$-monotone, so we make a further deformation of $(\qt_i(t))_{i \in \I}$ to get a path $ (\qt^{\ep}_i(t))_{i \in \I}, t \in [0,1/2]$ as follows 
\begin{align*}
&\qt^{\ep}_{i}(t) = 
\begin{cases}
\qe_i(t),\;\; \forall t \in [0, t_0];\\
\R_{xz}\qe_i(t)+2y_i^{\ep}(t_0)\e_2, \;\; \forall t \in [t_0, 1/2], 
\end{cases}
\;\; &&\text{if } i \in \I_{10}, \\
&\qt^{\ep}_i(t) = \qe_i(t)+ \yt^{\ep}_{k/2}(1/2) \e_2, \;\; \forall t \in [0, 1/2],\;\; &&\text{if } i \in \I_{11}, \\
&\qt^{\ep}_i(t) = \qe_i(t), \;\; \forall t \in [0, 1/2], \;\; &&\text{if } i \in \I_{12}, \\
&\qt^{\ep}_i(t) = \qe_i(t)+ \yt^{\ep}_{k/2+n}(1/2)\e_2, \;\; \forall t \in [0, 1/2], \;\; &&\text{if } i \in \I_{13}.  
\end{align*}
Figure \ref{fig:dfm3}(a) shows how $\qt^{\ep}_i, i \in \I_{10}$ are defined. As usual, each $\qt^{\ep}_i(t), i \in \N \setminus \I$ will be defined through the symmetric constraints. By the way $\qt^{\ep}(t)=(\qt^{\ep}_i(t))_{i \in \N}$ is defined, it belongs to $\lmn_{\le, \om}$. Meanwhile by a similar argument as in the proof of Lemma \ref{lm:B1}, we get $\A_{1/2}(\qt) < \A_{1/2}(\qe)< \A_{1/2}(q)$, which is a contradiction. This finishes our proof of \emph{Case 1}. 

\emph{Case 2}: $t_0=0$. The action of $g_2$ implies 
$$ (x_{k/2}, y_{k/2}, z_{k/2})(0) = (x_{n -k/2}, -y_{n-k/2}, z_{n-k/2})(0);$$
$$ (x_{k/2+n}, y_{k/2+n}, z_{k/2+n})(0) = (x_{2n -k/2}, -y_{2n-k/2}, z_{2n-k/2})(0).$$
Recall that $q_{k/2}(0) = q_{k/2+n}(0)$ is contained in the $x$-axis. Therefore 
$$ q_{k/2}(0) = q_{2n-k/2}(0) = q_{k/2+n}(0) = q_{n-k/2}(0). $$
Like in \emph{Case 1}, the strictly $x$-monotone constraints guarantee no other mass can be involved in this collision cluster, so $q$ has an isolated $\I_{20}$-cluster collision at $t=0$, where $\I_{20}=\{k/2, k/2+n, 2n-k/2, n-k/2 \}$. The rest $\I_{2j}$'s are defined following the same principle as in \emph{Case 1}. To be precise, $\I_{21}=\I_{11}, \I_{23}=\I_{13}$ and $\I_{22} = \I_{12} \setminus \{ 2n-k/2, n-k/2\}$. 

Now choose a $\tau=(\tau_i)_{i \in \I}$ satisfying
$$ \tau_{k/2}= \tau_{2n-k/2}=1, \; \tau_{k/2+n} = \tau_{n-k/2} =-1, \; \text{ and }\; \tau_i =0, \; \forall i \in \I \setminus \I_{20}. $$
With such a $\tau$, we get a $(\qe_i(t))_{i \in \I}, t \in [0,1/2]$, by deforming $(q_i(t))_{i \in \I}$ locally near $t=0$ using Lemma \ref{lm:dfm4} (see Figure \ref{fig:dfm3}(b)). This new path clearly is not $y$-monotone, so we further deform it as follows to get a $(\qt^{\ep}_i(t))_{i \in \I},$ $t \in [0,1/2]$:
\begin{align*}
&\qt^{\ep}_i(t) = \R_{xz}\qe_i(t)+2y_i^{\ep}(0)\e_2,\;\; &&\forall t \in [0,1/2], \;\; &&\text{if } i \in \I_{10}, \\
&\qt^{\ep}_i(t) = \qe_i(t)+ \yt^{\ep}_{k/2}(1/2)\e_2, \;\; &&\forall t \in [0,1/2],\;\; &&\text{if } i \in \I_{21}, \\
&\qt^{\ep}_i(t) = \qe_i(t), \;\; && \forall t \in [0,1/2], \;\; && \text{if } i \in \I_{22} \cup (\I_{20} \setminus \I_{10}), \\
&\qt^{\ep}_i(t) = \qe_i(t)+ \yt^{\ep}_{k/2+n}(1/2)\e_2, \;\; && \forall t \in [0,1/2], \;\; &&\text{if } i \in \I_{23}.
\end{align*}
Again see Figure \ref{fig:dfm3}(b) for an illuminating picture. The rest is just like \emph{Case 1}. 

\emph{Case 3}: $t_0=1/2$. The action of $g_2$ implies 
$$ (x_{k/2}, y_{k/2}, z_{k/2})(1/2) = (x_{n-1-k/2}, -y_{n-1-k/2}, z_{n-1-k/2})(1/2);$$
$$ (x_{k/2+n}, y_{k/2+n}, z_{k/2+n})(1/2) = (x_{2n-1-k/2}, -y_{2n-1-k/2}, z_{2n-1-k/2})(1/2).$$
As $q_{k/2}(1/2) = q_{k/2+n}(1/2)$ is contained in the $x$-axis, we have
$$ q_{k/2}(1/2) = q_{2n-1-k/2}(1/2) = q_{k/2+n}(1/2)= q_{n-1-k/2}(1/2). $$
Like in \emph{Case 2}, as $q$ is strictly $x$-monotone, it has an isolated $\I_{30}$-cluster collision at $t=1/2$, where $\I_{30}=\{k/2, k/2+n, 2n-1-k/2, n-1-k/2\}$. Furthermore we let $\I_{31}= \I_{11} \setminus \{2n-1-k/2\}$, $\I_{32}= \I_{12}$ and $\I_{33}= \I_{13} \setminus \{n-1-k/2\}.$

Choose a $\tau=(\tau_i)_{i \in \I}$ satisfying 
$$ \tau_{k/2}= \tau_{2n-1-k/2}=1, \; \tau_{k/2+n} = \tau_{n-1-k/2}=-1, \; \text{ and } \; \tau_i=0, \; \forall i \in \I \setminus \I_{30}. $$
With such a $\tau$, we get a new path $(\qe_i(t))_{i \in \I}, t \in [0, 1/2]$, by deforming $(q_i(t))_{i \in \I}$ locally near $t=1/2$ using Lemma \ref{lm:dfm5} (see Figure \ref{fig:dfm3}(c)). 

Again such a $q^{\ep}$ is not $y$-monotone, so we further deform it as follows to get a $(\qt^{\ep}_i(t))_{i \in \I}$, $t \in [0,1/2]$:
\begin{align*}
& \qt^{\ep}_i(t)= \R_{xz}\qe_i(t) + 2\qe_i(1/2)\e_2, \;\; &&\forall t \in [0,1/2], \;\; && \text{if } i \in \I_{30} \setminus \I_{10}, \\
& \qt^{\ep}_i(t)=\qe_i(t) + \yt^{\ep}_{k/2}(0) \e_2, \;\; &&\forall t \in [0,1/2], \;\; && \text{if } i \in \I_{31}, \\
& \qt^{\ep}_i(t)= \qe_i(t),\;\; &&\forall t \in [0,1/2], \;\; && \text{if } i \in \I_{32} \cup \I_{10}, \\
& \qt^{\ep}_i(t) = \qe_i(t)+\yt^{\ep}_{n-1-k/2}(0) \e_2, \;\; &&\forall t \in [0,1/2], \;\; && \text{if } i \in \I_{33}. 
\end{align*}
See Figure \ref{fig:dfm3}(c) for a picture. The rest is just like \emph{Case 1}. 

This finishes our proof of $k$ being even. Now let's assume $k$ is odd, by \eqref{eqn:g1} and \eqref{eqn:tc},
\begin{align*}
&z_{2n-1-[k/2]}(1/2) = z_0(k/2) = \om_k |z_0(k/2)| = -|z_0(k/2)| < 0;\\
&z_{2n-1-[k/2]}(0) = z_0(\frac{k+1}{2}) = \om_{k+1}|z_0(\frac{k+1}{2})| = |z_0(\frac{k+1}{2})| >0.
\end{align*}
This implies $z_{2n-1-[k/2]}(t_0)=0$, for some $t_0 \in [0,1/2]$. By the action of $h_1$, 
$$ q_{n-1-[k/2]}(t) =  \mf{R}_{x}q_{2n-1-[k/2]}(t), \;\; \forall t. $$
This shows $z_{n-1-[k/2]}(t_0)= -z_{2n-1-[k/2]}(t_0)=0$. Meanwhile by our assumption $ y_{n-1-[k/2]}(t_0) = y_{2n-1-[k/2]}(t_0) =0$. As a result, $q_{n-1-[k/2]}(t_0)= q_{2n-1-[k/2]}(t_0)$ belong to the $x$-axis, and $q$ has an isolated collision at the moment $t_0$ involving at least $m_{n-1-[k/2]}$ and $m_{2n-1-[k/2]}$. By similar arguments as above, we can show $q$ has an isolated $\I_0$-cluster collision at $t_0$, where 
$$ \I_0= \begin{cases}
\{n-1-[k/2], 2n-1-[k/2]\}, \;\; &\text{ if } t_0 \in (0, 1/2), \\
\{ 1+[k/2], n+1+[k/2], n-1-[k/2], 2n-1-[k/2] \}, \;\; & \text{ if } t_0=0, \\
\{[k/2], n+[k/2], n-1-[k/2], 2n-1-[k/2]\}, \;\; & \text{ if } t_0=1/2. 
\end{cases}
$$
In each of those three cases, a contradiction can be reached by a similarly argument used in the corresponding case with $k$ being even. We omit the details here. 
\end{proof}

\mbox{}

\emph{Acknowledgements.} The author wishes to express his gratitude to Kuo-Chang Chen, Jacques F\'ejoz and Richard Montgomery for their interests and comments on this work. He thanks the anonymous referee for a careful reading of the paper and useful suggestions. He thanks the supports of Ke Zhang through NSERC at University of Toronto and Xijun Hu through NSFC(No. 11425105) at Shandong University, where most of the work was done.


\bibliographystyle{abbrv}
\bibliography{ref-DoubleChore}

\end{document}